\documentclass[11pt]{amsart}

\title{A Bennequin-type inequality and combinatorial bounds}
\author[C.\ Collari]{Carlo Collari}
\address{IMT, Universit\'{e} Paul Sabatier\\
118, route de Narbonne\\
F-31062 Toulouse Cedex 9}
\email{collari@math.unifi.it}
\urladdr{http://poisson.phc.unipi.it/~collari}

\usepackage[sc]{mathpazo}
\usepackage[utf8]{inputenc}
\usepackage[T1]{fontenc}
\usepackage[UKenglish]{babel}

\usepackage{amsmath}
\usepackage{amsfonts}
\usepackage{amssymb}
\usepackage{amsthm}

\usepackage{geometry}
\usepackage{float}
\usepackage{tikz}
\usetikzlibrary{arrows,matrix,patterns,decorations.markings}
\usepackage{float}
\usepackage[all,knot]{xy}

\usepackage{mathrsfs}
\usepackage{wasysym}
\usepackage{enumerate}

\theoremstyle{plain}
	\newtheorem{theorem}{\noindent\textbf{Theorem}}
	\newtheorem*{theorem*}{\noindent\textbf{Theorem}}
	\newtheorem{lemma}[theorem]{\noindent\textbf{Lemma}}
	\newtheorem{proposition}[theorem]{\noindent\textbf{Proposition}}
	\newtheorem*{proposition*}{\noindent\textbf{Proposition}}
	\newtheorem{corollary}[theorem]{\noindent\textbf{Corollary}}
	\newtheorem*{corollary*}{\noindent\textbf{Corollary}}

\theoremstyle{definition}
	\newtheorem{definition*}{\noindent\textbf{Definition}}

\theoremstyle{remark}
	\newtheorem{rem*}{\noindent\textbf{Remark}}
	\newtheorem*{warning}{\noindent\textbf{Warning}}
	
\theoremstyle{definition}
	\newtheorem{definition}{\noindent\textbf{Definition}}
	\newtheorem{example}{\noindent\textbf{Example}}
	\newtheorem{question}{\noindent\textbf{Question}}

\newcommand{\ocrosstext}{\raisebox{-0.1cm}{\begin{tikzpicture}[scale=.07]
\begin{scope}[shift={(2,2)}]
\draw (+2,-2) -- (-2,+2);
\pgfsetlinewidth{8*\pgflinewidth}
\draw[white] (-2,-2) -- (+2,+2);
\pgfsetlinewidth{.125*\pgflinewidth}
\draw (-2,-2) -- (+2,+2);
\end{scope}
\draw[dashed] (2,2) circle (2.8cm);
\end{tikzpicture}}\:}

\newcommand{\isplittext}{\raisebox{-0.1cm}{\begin{tikzpicture}[scale=.07]
\begin{scope}[shift={(2,2)}]
\draw (-2,-2) .. controls +(1,1) and +(1,-1) ..  (-2,+2);
\draw  (2,-2).. controls +(-1,1) and +(-1,-1) ..  (2,+2);
\end{scope}
\draw[dashed] (2,2) circle (2.8cm);
\end{tikzpicture}}\:}

\newcommand{\osplittext}{\raisebox{-0.1cm}{\begin{tikzpicture}[scale=.07]
\begin{scope}[shift={(2,2)}]
\draw (-2,-2) .. controls +(1,1) and +(-1,1) ..  (2,-2);
\draw (-2,+2) .. controls +(1,-1) and +(-1,-1) ..  (2,+2);
\end{scope}
\draw[dashed] (2,2) circle (2.8cm);
\end{tikzpicture}}\:}

\begin{document}

\begin{abstract}
In this paper we provide a new Bennequin-type inequality for the Rasmussen-Beliakova-Wehrli invariant, featuring the numerical transverse braid invariants (the $c$-invariants) introduced by the author. From the Bennequin type-inequality, and a combinatorial bound on the value of the $c$-invariants, we deduce a new computable bound on the Rasmussen invariant.
\end{abstract}
\maketitle
\section{Introduction}

\subsection*{Motivations and statement of results}

A contact structure on a $3$-dimensional manifold is a totally non-integrable plane distribution. One of the most basic examples of contact structure is the \emph{symmetric contact structure} in $\mathbb{R}^3$, that is to say
\[ \xi_{sym} = Ker(dz - xdy + ydx),\]
where $(x,y,z)$ is a fixed system of coordinates in $\mathbb{R}^3$. The presence of a contact structure allows one to distinguish special families of knots and links. In particular, one can define the \emph{transverse links} to be all the smooth links which are nowhere tangent to $\xi_{sym}$. The study of transverse links is a central topic in low-dimensional and contact topology (see \cite{Etnyre05} for an overview).

Two transverse links are considered to be \emph{equivalent} if they can be moved one onto the other via an ambient isotopy, in such a way that at each instant the transversality condition is preserved. Transverse links have two \emph{classical invariants}: the (smooth) link-type, and the \emph{self-linking number} $sl$. The latter is a numerical invariant of the equivalence class of a transverse link. Finding invariants (\emph{effective invariants}) which can distinguish between transverse links with the same classical invariants is an important problem in low-dimensional contact topology.

In \cite{TransFromKhType17} the author introduced two numerical invariants, called $c$-invariants, for transverse links presented as closed braids (see \cite{Bennequin83, OrevkovShev03, Wrinkle03} for more details on this presentation). These invariants, denoted by $c_{\mathbb{F}}$ and $\overline{c}_{\mathbb{F}}$, are two non-negative integers. It is not known whether or not these invariants are effective. A proof of them being non-effective would imply that (the vanishing of) the Plamenevskaya invariant $\psi$ is non-effective, which is an open problem at the time of writing. Applications aside, it is pretty difficult in general to compute the $c$-invariants. One of the main results of this paper is the introduction of a lower bound on the values of the $c$-invariants. Even if there is no equality in general, this bound can be used to compute the $c$-invariant in a number of cases.

Before stating the result it will be necessary to fix some terminology and notation. This notation will be used throughout the paper. Let $D$ be an oriented link diagram. It is possible to associate to $D$ a graph $\Gamma(D)$, which we call the \emph{simplified Seifert graph}. The vertices of $\Gamma(D)$ are the circles in the oriented resolution of $D$, and there is an edge between two vertices if the corresponding circles shared at least a crossing in $D$. The edges of the simplified Seifert graph can be divided into three classes: (1) \emph{positive}, if the circles corresponding to the endpoints of the edge shared only positive crossings, (2) \emph{negative} if the circles corresponding to the endpoints of the edge shared only negative crossings, and (3) \emph{neutral}, if the edge is neither positive nor negative. A vertex $v$ of $\Gamma(D)$ is called \emph{positive} (resp. \emph{negative}) if there are only positive (resp. negative) edges incident in $v$. A vertex which is neither positive nor negative is called \emph{neutral}. We will call \emph{pure} each vertex which is connected only to vertices of the same type as itself.
\begin{rem*}
An edge incident in a neutral vertex is not necessarily neutral. 
\end{rem*}
Finally, given an oriented link diagram $D$ we define the following quantities:
\begin{enumerate}
\item $V(D)$ = the number of vertices in $\Gamma(D)$;
\item $V_{\pm}(D)$ = the number of positive (resp. negative) vertices in $\Gamma(D)$;
\item $\ell_{\pm}(D)$ = the number of connected components of $\Gamma(D)$ composed only of positive (resp. negative) vertices;
\item $s_{\pm}(D)$ = the number of connected components of the graph obtained from $\Gamma(D)$ by removing all the negative (resp. positive) edges
\item $
\delta_{-}(D) = \begin{cases} 1 & \text{if there is a negative edge between two neutral vertices in }\Gamma(L) \\ 0 &\text{otherwise}\end{cases}$
\end{enumerate}

Now we can state our bound, which will be proved in Subsection \ref{Combinatorialbounds}.

\begin{theorem}\label{thm:main1}
Let $B$ be a braid and let $\mathbb{F}$ be any field. Then, the following inequality holds
\[\overline{c}_{\mathbb{F}}(B),\: c_{\mathbb{F}}(B) \geq V_{-}(B) - \ell_{-}(B) + \delta_{-}(B),\]
where $\overline{c}_{\mathbb{F}}(B)$ and $c_{\mathbb{F}}(B)$ denote the $c$-invariants of the braid $B$ over the field $\mathbb{F}$.
\end{theorem}

A lower bound is not, by itself, useful to compute a numerical invariant; so there is also the necessity of an upper bound. An upper bound for the $c$-invariants can be deduced from the following inequality, which is our second result (whose proof can be found in Subsection \ref{sec:Bennequinineq}).

\begin{theorem}[Bennequin-type inequality]\label{thm:main2}
Let $\lambda$ be an oriented link type, and let $T$ be a transverse representative of $\lambda$. Then, for each field $\mathbb{F}$, the following inequality holds
\[ s(\lambda; \mathbb{F}) \geq sl(T) + 2 c_{\mathbb{F}}(T) + 1,\]
where $s(\lambda;\mathbb{F})$ is the Rasmussen-Beliakova-Werhli (RBW) $s$-invariant of $\lambda$ (\cite{Rasmussen10, BeliakovaWehrli08}).
\end{theorem}  

The inequality in Theorem \ref{thm:main2} is part of a family of inequalities relating topological and transverse invariants. These bounds are named collectively \emph{Bennequin-type inequalities}. In particular, our inequality sharpens a similar result due, independently, to O. Plamenvskaya (\cite{Plamenevskaya06}) and A. Shumakovitch (\cite{Shumakovitch04}) which is the following 
\[ s(\lambda; \mathbb{F}) \geq sl(T) + 1.\]

Finally, proving Theorems \ref{thm:main1} and \ref{thm:main2} in to a slightly more general setting, and putting them together we obtain our third result, which is a combinatorial bound on the value of the $s$-invariant. This will be done in Subsection \ref{Combinatorialbounds}.

\begin{theorem}\label{thm:main3}
Let $D$ be an oriented link diagram representing the oriented link $L$. Then, we have the following bound
\begin{equation}
\tag{Cbound}
 s(L;\mathbb{F}) \geq w(D) - V(D) + 2 V_-(D) - 2 \ell_{-}(D) + 2\delta_{-}(D) + 1.
 \label{combinatorialboundons}
\end{equation}
where $w(D)$ indicates the \emph{writhe} of $D$ (i.e. the number of positive crossings minus the number of negative crossings). Moreover, the bound is sharp for negative and positive diagrams.
\end{theorem}

\begin{rem*}
A special case of \eqref{combinatorialboundons} without the $\delta_{-}$-term has been proved for non-positive and non-negative non-split link diagrams by Kawamura \cite{Kawamura15}. However, we remark that Kawamura's bound needs the non-split and the non-positivity/non-negativity hypotheses, while our bound does not. Furthermore, without the $\delta_{-}$-term Kawamura's bound is strictly less efficient than the bound in \eqref{KawCavbound} (due to Kawamura and extended for a special class of link by Cavallo). Finally, we remark that also \eqref{KawCavbound} needs some hypotheses either on the diagram (e.g. the diagram should be divided into non-split components) or on the link (e.g. the link should be pseudo-$Kh$-thin), and both these hypotheses can be arbitrarily hard to verify on a given diagram.
\end{rem*}

Recall that an oriented link diagram is \emph{almost-positive} if it has only a negative crossing, and at least a positive crossing. A link is \emph{almost-positive} if it admits an almost-positive diagram. As an application of Theorem \ref{thm:main3}, using the computations of the Seifert genus of almost-positive links due to Stoimenov (\cite{Stoimenov11}), we are able to recover a result due to Tagami (\cite{Tagami14}), in the case of knots, and Abe-Tagami (\cite{AbeTagami17}), general case.

\begin{corollary}[\cite{AbeTagami17,Tagami14}]\label{cor:AbeTagami}
Let $L$ be an almost-positive non-split link. Then, we have that
\[ s(L,\mathbb{F}) = 2g_{4}(L) + \ell(L) - 1 = 2g_{3}(L) + \ell(L) - 1,\]
where $\ell$ denotes the number of components, $g_3$ denotes the three-dimensional (Seifert) genus, and $g_{4}$ denotes the weak slice genus. Moreover, they can be computed directly from an almost-positive diagram $D$, and we have
\[ s(L,\mathbb{F}) = w(D) - V(D) + 2 \delta_{-}(D) + 1.\]
\end{corollary}

In particular, all almost-positive non-split diagrams are such that \eqref{combinatorialboundons} is sharp, while the lower bound in \eqref{KawCavbound} is never sharp for all almost-positive non-split diagrams such that $\delta_{-} \ne 0$. This immediately implies the following result

\begin{theorem}\label{thm:independence}
The bounds \eqref{combinatorialboundons} and the lower bounds in \eqref{KawCavbound} and \eqref{Lobbbound} are independent. More precisely there are diagrams for which one bound is sharper than the other.\qed
\end{theorem}

Both Corollary \ref{cor:AbeTagami} and Theorem \ref{thm:independence} will be proved in Subsection \ref{sec:examples}.

\begin{rem*}
There is a larger class of diagrams including the almost-positive ones for which \eqref{combinatorialboundons} is sharp, while \eqref{KawCavbound} is not.  For instance, the reader may consult \cite[Chapter 4, Subsection 2.5]{Thesis1}. However, for the family in \cite{Thesis1} and for almost-positive links the bound \eqref{combinatorialboundons} can be re-proved (almost trivially) from Tagami and Abe-Tagami results. In Subsection \ref{sec:examples}, we will give yet another family of links for which \eqref{combinatorialboundons} is sharper than the lower bounds in \eqref{KawCavbound} and \eqref{Lobbbound}, and for which our result does not follow trivially from Tagami's result.
\end{rem*}

\subsection*{Acknowledgments}
The author wish to thank Prof. Paolo Lisca for his advice, the helpful conversations and his continuous support. Many thanks are due to Daniele Celoria, Marco Moraschini, and Stefano Riolo for their comments on the drafts of this paper. The author also wish to thank an anonymous referee for pointing out the reuslts of Tagami and Abe-Tagami, and the other referees for the helpful remarks. The present paper is an excerpt from the author's Ph.D. thesis. During his Ph.D. the author was supported by a Ph.D. scholarship ``Firenze-Perugia-Indam''. The author is currently supported by an Indam grant for a research period outside Italy, and is currently hosted by the IMT.

\section{Deformations of Khovanov homology and concordance invariants}
In this section we will briefly review some well-known material concerning Khovanov homology and its deformations.
\subsection{Frobenius Algebras}
Let $R$ be a (commutative unital) ring. A \emph{Frobenius algebra} (FA) over $R$ is a commutative unital $R$-algebra $A$, which is also a projective $R$-module, together with two $R$-linear maps
\[\Delta: A \longrightarrow A \otimes_{R} A\qquad\text{and}\qquad \epsilon :A \longrightarrow R,\]
called \emph{co-multiplication} and \emph{co-unit}, respectively, satisfying the following properties: (1) $\Delta$ is co-commutative and co-associative, (2) $\Delta$ commutes with the left and right action of $A$, and (3) $\epsilon$ is the co-unit relative to $\Delta$ (that is, $\epsilon \otimes Id_{A}\circ\Delta   =  Id_{A} \otimes\epsilon\circ\Delta = Id_{A} $). We give some examples of FA below. These examples will be needed during the rest of the paper.

\begin{example}[Khovanov FA]
Let $R$ be a ring. Then, the \emph{Khovanov Frobenius Algebra} is the $R$-algebra $A_{Kh} = R[X]/(X^2)$ endowed with the co-multiplication
\[\Delta_{Kh}(1) = X \otimes 1 + 1 \otimes X\qquad \Delta_{Kh}(X) = X \otimes X,\]
and co-unit
\[\epsilon_{Kh}(1) = 0\qquad \epsilon_{Kh}(X)=1. \]
\end{example}

\begin{example}[Twisted Lee FA]
Let $R$ be a ring. Then, the \emph{twisted Lee Frobenius Algebra} is the $R$-algebra $A_{TLee} = R[X]/(X^2 - X)$ endowed with the co-multiplication
\[\Delta_{TLee}(1) = X \otimes 1 + 1 \otimes X - 1 \otimes 1 \qquad \Delta_{TLee}(X) = X \otimes X,\]
and co-unit
\[\epsilon_{TLee}(1) = 0\qquad \epsilon_{TLee}(X)=1. \]
\end{example}

\begin{example}[Bar-Natan FA]
Let $R$ be a ring and $U$ a formal variable. Then, the \emph{Bar-Natan Frobenius Algebra} is the $R_{BN}$($=R[U]$)-algebra $A_{BN} = (R[U])[X]/(X^2 - UX)$ endowed with the co-multiplication
\[\Delta_{BN}(1) = X \otimes 1 + 1 \otimes X - U \cdot 1 \otimes 1 \qquad \Delta_{BN}(X) = X \otimes X,\]
and co-unit
\[\epsilon_{BN}(1) = 0\qquad \epsilon_{BN}(X)=1. \]
\end{example}

Certain Frobenius algebras may be endowed with further structure. A Frobenius algebra is called \emph{graded} (resp. \emph{filtered}) if $R$ is a graded ring, $A$ is a graded (resp. filtered) $R$-module, and the multiplication, co-multiplication, unit and co-unit are graded (resp. filtered) maps. Where the grading (resp. filtration) on $A\otimes_{R} A$ is the usual tensor grading (resp. filtration).

By taking the ring $R$ to be trivially graded, and by setting
\[deg(X) = -2\qquad deg(1) = 0\qquad deg(U) = -2,\]
Khovanov FA and Bar-Natan FA become graded FA's, while twisted Lee theory becomes a filtered FA. The filtration in twisted Lee theory is obtained by setting
\[\mathscr{F}_{-3} = (0),\quad\mathscr{F}_{-2} = \mathscr{F}_{-1} =\langle  X \rangle_{R} \subseteq A_{TLee}\quad \text{and} \quad \mathscr{F}_{0} = A_{TLee}.\]
In this way $1$ and $X$ become a filtered basis for $A_{TLee}$.

\begin{rem*}\label{gradingsFA}
Notice that Khovanov and twisted Lee FA's can be seen as quotients of Bar-Natan FA. Moreover, since $A_{Kh}$ is obtained as the quotient of $A_{BN}$ by an homogeneous ideal, it inherits a grading. This grading is the same as the grading on $A_{Kh}$ we defined. Similarly, one can see $A_{TLee}$ as a quotient of $A_{BN}$. However, this time is the quotient by the ideal generated by $(U-1)$. Since, this ideal is not homogeneous, twisted Lee theory does not inherit a grading. However, the quotient has a natural filtration (cf. \cite[Appendix A, Section 2]{Thesis1}), which coincides with the filtration on twisted Lee FA defined above.
\end{rem*}

\subsection{From Frobenius algebras to link homology}

Let $D$ be an oriented link diagram. A \emph{local resolution} of a crossing \ocrosstext in $D$ is its replacement with either a $0$-resolution \isplittext or with a $1$-resolution \osplittext .

\begin{definition}
A \emph{resolution of} $D$ is the set of circles, embedded in $\mathbb{R}^{2}$, obtained from $D$ by performing a local resolution at each crossing. The total number of $1$-resolutions performed in order to obtain a resolution $\underline{s}$ will be denoted by $\vert \underline{s}\vert$.
\end{definition}

A circle $\gamma$ in a resolution \emph{touches a crossing} $c$ if changing the local resolution at $c$ either splits $\gamma$ or merges it with another circle. Finally, two (possibly equal) circles are \emph{joined by}, or  \emph{share}, a crossing $c$ if they are the only circles touching $c$.

\begin{rem*}
A crossing touches at most two circles.
\end{rem*}

Let $\mathcal{R}_D$ be the set of all the possible resolutions of $D$. Given a graded (resp. filtered) Frobenius algebra, say $\mathcal{F} = (R,A,\iota, \Delta,\epsilon)$, define
\[ C_{\mathcal{F}}^{i}(D,R) = \bigoplus_{\scriptsize{\begin{array}{c} \underline{r}\in \mathcal{R}_{D}\\
\vert \underline{r} \vert - n_{-} = i
\end{array}}} A_{\underline{r}}\qquad \text{and}\qquad A_{\underline{r}} = \bigotimes_{\gamma\in \underline{r}} A_{\gamma}(1), \]
where $A_{\gamma}$ is just an indexed copy of $A$, $n_{-}$ denotes the number of negative crossings in $D$, and $(\cdot)$ denotes the degree (resp. filtration) shift. (The shift can be safely ignored until the next subsection.) These are the $R$-modules that are going to play the role of (co)chain groups. So, in order to define a (co)chain complex, all that is left to do is to define a differential. This is done in two steps: first, one defines the maps
\[ d_{\underline{r}}^{\underline{s}}: A_{\underline{r}} \to A_{\underline{s}},\quad \forall\ \underline{r},\: \underline{s}\in \mathcal{R}_{D},\]
and $\underline{r}$ differs from $\underline{s}$ only for the replacement of a $0$-resolution with a $1$-resolution.

Consider $x= \bigotimes_{\gamma\in \underline{r}} \alpha_{\gamma} \in A_{\underline{r}}$. Then taken a resolution $\underline{s}$ as above, there is an identification of all the circles of $\underline{r}$ and $\underline{s}$, except the ones involved in the change of local resolution. There are only two cases to consider: (a) two circles of $\underline{r}$, say $\gamma_{1},\ \gamma_{2}$ are merged in a single circle $\gamma^\prime_{1}$ in $\underline{s}$, or (b) a circle $\gamma_{1}$ belonging to $\underline{r}$ is split in into two circles, say $\gamma_{1}^\prime$ and $\gamma_{2}^\prime$, in $\underline{s}$. Our map is defined as follows
\[ d_{\underline{r}}^{\underline{s}} (x) = \begin{cases}\bigotimes_{\gamma\in \underline{r}\cap \underline{s}} \alpha_{\gamma} \otimes m(\alpha_{\gamma_{1}},\alpha_{\gamma_{2}}) & \text{in case (a)}\\
\bigotimes_{\gamma\in \underline{r}\cap \underline{s}} \alpha_{\gamma} \otimes \Delta(\alpha_{\gamma_{1}}) & \text{in case (b)} \\\end{cases}\]
Notice that $d_{\underline{r}}^{\underline{s}}$ is well defined because of the commutativity of $m$, and of the co-commutativity of $\Delta$. 

Then, one defines
\[ d_{\mathcal{F}}^{i}: C^{i}_{\mathcal{F}}(D,R) \to C^{i+1}_{\mathcal{F}}(D,R): x\in A_{\underline{r}} \mapsto \sum_{\underline{r}\prec\underline{s}} S(\underline{r},\underline{s})d_{\underline{r}}^{\underline{s}}(x).\]
where $S(\underline{r},\underline{s})$ is an appropriate signage such that $d_{\mathcal{F}}^{2} = 0$.

This construction depends on a number of choices (and on the existence of a sign function $S$). The following theorem ensures that everything is well defined, and that the homology of $C_{\mathcal{F}}^{\bullet}(D,R)$ does indeed exist.

\begin{theorem}[Khovanov \cite{Khovanov00}, Lee \cite{Lee05}, Bar-Natan \cite{BarNatan05cob}]
Let $\mathcal{F}$ be one among $Kh$, $TLee$ and $BN$. There exists a sign function $S$ such that the complex $(C_{\mathcal{F}}^{\bullet}(D,R),d^{\bullet}_{\mathcal{F}})$ is a (co)chain complex. This complex does not depend, up to isomorphism, on the choice of the sign function $S$, or on the order of the circles in each resolution. 
\qed
\end{theorem}

\subsection{Gradings and filtrations}\label{Subs:GradingC_F}

Since $Kh$ and $BN$ are graded FA and $TLee$ is a filtered FA, it is possible to endow $(C_{\mathcal{F}}^{\bullet}(D,R),d^{\bullet}_{\mathcal{F}})$ with further structure; more precisely, if $\mathcal{F}$ is $Kh$ or $BN$, then we can define a second grading, and if $\mathcal{F}$ is $TLee$ then we can define a filtration.

Before proceeding any further let us fix some notation. The elements of $C^{i}_{\mathcal{F}}(D,R)$ of the form $\bigotimes_{\gamma\in\underline{r}} \alpha_{\gamma}$, with $\alpha_{\gamma} \in \{ 1, X\}$ and $\underline{r}\in \mathcal{R}_D$, will be called \emph{states}; while those of the form $\bigotimes_{\gamma\in\underline{r}} \alpha_{\gamma}$, with $\alpha_\gamma \in A$, will be called \emph{enhanced states}. Notice that the states are an $R$-basis of $C^{\bullet}_{\mathcal{F}}(D,R)$, while the enhanced states are a system of generators.

Let $\mathcal{F}$ be either $Kh$ or $BN$. The \emph{quantum grading} over the complex $(C^{\bullet}_{\mathcal{F}}(D,R),d^{\bullet}_{\mathcal{F}})$, is defined as follows
\[ qdeg(\bigotimes_{\gamma\in \underline{r}} \alpha_{\gamma}) = \sum_{\gamma\in\underline{r}} deg_{A}(\alpha_{\gamma}) - 2 n_{-} + n_{+}+  \vert \underline{r} \vert,\]
for each state $\bigotimes_{\gamma\in\underline{r}} \alpha_{\gamma}$.

\begin{rem*}
For a non homogeneous element $x$ we set $qdeg(x)$ to be the \emph{minimum} degree among its homogeneous components.
\end{rem*}

Now consider the Frobenius algebra $TLee$. The \emph{quantum filtration} over the complex $(C^{\bullet}_{\mathcal{F}}(D,R),d^{\bullet}_{\mathcal{F}})$, is defined on each homological grading as the (direct sum of the) tensor filtration shifted by $- 2 n_{-} + n_{+}$. We will denote by $\mathscr{F}_{\circ}$ the filtration induced on the complex $C_{TLee}^\bullet(D,R)$. Whenever it is necessary to specify the dependence on the diagram or on the base ring we will write $\mathscr{F}_{\circ}(D,R)$ instead of $\mathscr{F}_{\circ}$. Moreover, if we wish to denote the filtration restricted to a precise homological degree we will write $\mathscr{F}_{\circ}C^{i}$. 

The induced filtration in homology will be denoted by $\mathscr{F}_{\circ}H^\bullet(D,R)$, or by $\mathscr{F}_{\circ}H^\bullet$; the latter notation will be used whenever the link and the base ring are clear form the context, or unimportant for the discussion at hand.

It is easy to see that the differential $d_{\mathcal{F}}^\bullet$ is homogeneous (resp. filtered) with respect to the $qdeg$ degree (resp. induced filtration), and the resulting homology theory is hence doubly-graded (resp. filtered).

\begin{theorem}[Khovanov \cite{Khovanov00}, Rasmussen \cite{Rasmussen10}, Bar-Natan \cite{BarNatan05cob}]
If $\mathcal{F}$ is $Kh$ or $BN$ (resp. $TLee$), the homology of $(C_{\mathcal{F}}^{\bullet}(D,R),d^{\bullet}_{\mathcal{F}})$ is a link invariant up to bi-graded (resp. filtered) isomorphism of $R$-modules.\qed
\end{theorem}

\section{The $s$-invariant}

\subsection{Definition and first properties} The Twisted Lee homology is not, by itself, a very interesting invariant; as a matter of fact, disregarding the filtration the Twisted Lee homology of an oriented link depends only on the linking matrix (i.e.\: on the number of components and on the linking numbers between them). More precisely, E. S. Lee  in \cite[Theorem 5.1]{Lee05} has shown that: fixed an oriented link diagram $D$ representing a link $L$ there is a set of cycles, called \emph{canonical generators}, whose homology classes generate $H_{TLee}^{\bullet}(L, \mathbb{F})$. This set is indexed by the possible orientations of the underlying unoriented diagram. Moreover, the homological degree of each canonical generator is completely determined by the linking matrix of $L$.

Let $D$ be an oriented link diagram, the set of the possible orientations of the underlying unoriented diagram will be denoted by $\mathbb{O}(D)$. Fix a field $\mathbb{F}$, we indicate with $\mathbf{v}_{o}(D;\mathbb{F}) \in C_{TLee}^{\bullet}(D,\mathbb{F})$ the canonical generator associated to $o\in \mathbb{O}(D)$. 

\begin{definition}[Rasmussen \cite{Rasmussen10}, Beliakova-Wehrli \cite{BeliakovaWehrli08}]
Let $D$ be an oriented link diagram representing an oriented link $L$. The \emph{Rasmussen-Beliakova-Wehrli} (\emph{RBW}) \emph{invariant} associated to $o\in \mathbb{O}(D)$ is the integer
\[ s(o,L; \mathbb{F}) = \frac{Fdeg\left([\mathbf{v}_{o}(D;\mathbb{F}) - \mathbf{v}_{-o}(D;\mathbb{F})\right]) - Fdeg\left([\mathbf{v}_{o}(D;\mathbb{F}) + \mathbf{v}_{-o}(D;\mathbb{F})]\right)}{2},\]
where $Fdeg$ indicates the filtered degree in $H_{TLee}^{\bullet}(L, \mathbb{F})$, and $-o$ denotes the opposite orientation with respect to $o$. If $o$ is exactly the orientation induced by $L$, we will omit $o$ from the notation and call $s(L;\mathbb{F})$ the \emph{$s$-invariant} of $L$.
\end{definition}

\begin{rem*}
The original definition of the RBW-invariants (\cite{Rasmussen10, BeliakovaWehrli08}) is not in terms of Twisted Lee theory but in terms of Lee theory (\cite{Lee05}). We preferred to use the twisted version of Lee theory because it allows a bit more generality; while two theories are equivalent if $char(\mathbb{F})\ne 2$ (in particular, the corresponding RBW-invariants over $\mathbb{F}$ are equal, see \cite{Mackaayturnervaz07}), in characteristic $2$ is not possible to define the RBW-invariants using the original version of Lee theory.
\end{rem*}

The most interesting properties of the RBW-invariants are related to concordance, so us recall some definitions. Let $L_0$ and $L_1$ be two oriented links in $\mathbb{R}^{3}$. A \emph{cobordism} between $L_0$ and $L_1$ is a compact oriented surface $\Sigma$, properly (and smoothly) embedded in $\mathbb{R}^{3}\times [0,1]$, such that
\[ \Sigma \cap \mathbb{R}^{3} \times \{ i \} = L_i,\qquad i\in \{ 0, 1\},\]
with the induced orientation on $L_0$ and the opposite of the induced orientation on $L_1$.

Given a cobordism $\Sigma$ between two links, say $L$ and $L^\prime$, it is useful to distinguish between its connected components; a component of $\Sigma$ is of \emph{first type} if it bounds a component of $L$, and is of \emph{second type} otherwise. Let $\widetilde{L}$ be the unoriented link underlying $L$. Two orientations $o$ and $o^\prime$, on $\widetilde{L}$ and $\widetilde{L^\prime}$, are \emph{compatible via $\Sigma$} if there exists an orientation of $\Sigma$ bounding the oriented links $(\widetilde{L},o)$ and $(\widetilde{L^\prime},-o^\prime)$. 

\begin{definition}
A cobordism $\Sigma$ is a \emph{weak cobordism} (or weakly connected in the language of \cite{Rasmussen10}) if all its components are of first type. While $\Sigma$ is a \emph{strong cobordism} if each component of $\Sigma$ bounds exactly one component of $L$ and one component of $L^\prime$.
\end{definition}

\begin{rem*}
Every strong cobordism is also a weak cobordism. For each weak cobordism between $L$ and $L^\prime$, there is a unique orientation of $\widetilde{L^\prime}$ compatible with the orientation of $L$. 
\end{rem*}

\begin{definition}
A link $L$ is \emph{strongly concordant} (resp. \emph{weakly concordant}) to a link $L^\prime$ if there exists a strong (resp. weak) cobordism of genus $0$ between $L$ and $L^\prime$. Any link which is strongly (resp. weakly) concordant to an unlink is called \emph{strongly} (resp. \emph{weakly}) \emph{slice}.
\end{definition}

The main properties enjoyed by the RBW-invariants are summarized in the following proposition.

\begin{proposition}[Rasmussen \cite{Rasmussen10}, Beliakova-Wehrli \cite{BeliakovaWehrli08}]\label{Prop:sproperties}
Let $L$ and $L_{0}$ be two oriented links in $\mathbb{R}^{3}$. Then, the following
properties hold
\begin{enumerate}
\item\label{case:boundonchi} if $\Sigma$ is a weak cobordism between $L$ and $L_{0}$, then $ \vert s(L;\mathbb{F}) - s(L_{0};\mathbb{F}) \vert \leq -\chi(\Sigma)$;
\item $s(L \sqcup L_{0}; \mathbb{F}) = s(L;\mathbb{F}) + s(L_{0};\mathbb{F}) - 1$;
\item if $L$ and $L_0$ are knots, then $s(L \# L_{0}; \mathbb{F}) = s(L;\mathbb{F}) + s(L_{0};\mathbb{F})$;
\item $ 2 - 2 \ell \leq s(L;\mathbb{F}) + s(m(L);\mathbb{F}) \leq 2$;
\item $s$ detects the slice genus (i.e. the smallest genus of a strong cobordism between a link and the unlink) of all positive torus links;
\end{enumerate}
where $\#$ denotes the connected sum, $m()$ denotes the mirror image and $\ell$ is the number of components of $L$. In particular, $s$ is a strong concordance invariant.
\end{proposition}

It follows from the previous proposition that the $s$-invariant defines an homomorphism between the concordance group of knots (the operation being the connected sum) and the integers. Moreover, the value of the $s$-invariant provides a lower bound on the slice genus. This led, for example, to a combinatorial proof of Milnor conjecture, and to the effective computation of the slice genus for (infinite) families of knots and links.

\subsection{Bounds on the value of $s$}\label{Boundsons}

Computing the value of the $s$-invariant is not easy in general. The computation of this invariant usually involves computing spectral sequences. Under certain suitable hypotheses (e.g. $Kh$-thinness) the computations are easy, and the $s$-invariant may be read almost immediately from Khovanov homology. However, these hypotheses are do not hold in general, especially in the case of links. This is the reason why some combinatorial bounds are needed; by \emph{combinatorial bounds} we mean bounds obtained by linear combination of quantities which can be computed almost directly from the diagram.

Let $D$ be an oriented link diagram representing an $\ell$-components oriented link $L$. The number of split components of $L$ will be denoted by $\ell_{s}$.

To the best of the author's knowledge all combinatorial bounds on the value of the $s$-invariant are either weaker or equivalent to two: an upper and a lower bound due to Lobb (\cite{Lobb11})
\begin{equation*}
\tag{Lobb12}
 w(D)  + V(D) - 2s_{-}(D) + 1 \geq s(L; \mathbb{F})  \geq w(D) - V(D) + 2 s_{+}(D)  - 2 \ell(D) + 1.
 \label{Lobbbound}
\end{equation*}
and a lower bound due, independently, to Kawamura (\cite{Kawamura15}), in the case of diagrams with non-splittable connected components, and to Cavallo (\cite{Alberto15}) for $Kh$-pseudo-thin links (i.e. links whose Khovanov homology in homological degree $0$ is concentrated in two quantum degrees)
\begin{equation}
\tag{KwC15}
s(L; \mathbb{F}) \geq w(D) - V(D) + 2 s_{+}(D)  - 2 \ell_s(D) + 1. 
\label{KawCavbound}
\end{equation}

\begin{rem*}
Even if \eqref{KawCavbound} is strictly stronger than the lower bound in \eqref{Lobbbound}, some hypotheses on either $D$ or $L$ are necessary. The $Kh$-pseudo-thinness does not hold for general families of links.  Even though any link diagram can be (theoretically) reduced to a diagram with non-splittable components, given a link diagram is not easy to understand if a connected component of the diagram is splittable or not. Moreover, it is even harder to reduce the diagram into one possessing only non-splittable components. 
\end{rem*}

\section{The $\beta$- and $c$-invariants}

The $\beta$-cycles, introduced in \cite{TransFromKhType17}, are cycles in the Bar-Natan chain complex of a given oriented link diagram $D$. These can be thought as homogeneous lifts of the canonical generators in twisted Lee theory. More precisely, to an oriented link diagram $D$ are associated two elements $\beta(D)$, $\overline{\beta}(D)\in C_{BN}(D,\mathbb{F}[U])$ such that: the projection onto the quotient
\[ \pi_{TLee}: C^\bullet_{BN}(D,\mathbb{F}[U]) \longrightarrow \frac{C^\bullet_{BN}(D,\mathbb{F}[U])}{(U-1)C^\bullet_{BN}(D,\mathbb{F}[U])} = C^\bullet_{TLee}(D,\mathbb{F}),\]
sends the $\beta$-cycles to the canonical generators, that is 
\begin{equation}
\pi_{TLee}(\beta(D)) = \mathbf{v}_{o_{D}}(D),\quad \pi_{TLee}(\overline{\beta}(D)) = \mathbf{v}_{-o_{D}}(D),
\label{eq:projectionofbeta}
\end{equation}
where $o_{D}$ is the orientation of $D$.

\begin{rem*}
The map $\pi_{TLee}$ induces an isomorphism of $\mathbb{F}$-vector spaces between $C_{BN}^{i,j}(D,\mathbb{F}[U])$ and $\mathscr{F}^{j}C^{i}(D,\mathbb{F})$.
\end{rem*}

In the light of the previous remark, Equation \eqref{eq:projectionofbeta} and
\begin{equation}
qdeg(\beta(D)) = qdeg(\overline{\beta}(D)) = w(D) - V(D),
\label{quantum_deg_beta}
\end{equation}
totally characterize $\beta(D)$ and $\overline{\beta}(D)$ (cf. \cite[Lemma 6.3 \& Proposition 6.4]{TransFromKhType17}). However, for our scope it will be necessary to have an explicit definition of $\beta$ (and $\overline{\beta}$) as enhanced states.

Let $D$ be an oriented link diagram, and fix an orientation $o \in \mathbb{O}(D)$. The enhanced state $\beta(o,R)$ has underlying resolution the oriented resolution corresponding to $o$, denote it by $\underline{r}_{o}$, and each circle $\gamma$ has label
\[ b_\gamma =
\begin{cases}
x_\circ = X & \text{if }N(\gamma,o)\equiv 0\ mod\: 2\\
x_ \bullet = X - U & \text{if }N(\gamma,o)\equiv 1\ mod\: 2 \\
\end{cases}
\]
where $N(\gamma,o)$ denotes the \emph{nesting number} of $\gamma$ w.r.t. $o$, which is defined as follows. Fix an orientation on the plane. Given an orientation $o\in \mathbb{O}(D)$, each circle in $\underline{r}_{o}$ is naturally oriented. Thus, for each point $p \in \gamma$, we can consider the half-line $h_p$ from $p$, whose direction is the positive normal direction w.r.t. the orientation of $\gamma$. $N(\gamma,o)$ is the modulo $2$ count of the intersections of $h_p$ with the circles in $\underline{r}_{o}$. It is easy to see $N(\gamma,o)$ that  does not depend on $p$, for the choice of $p$ in a dense open set of $\gamma$, and thus is well defined.

\begin{rem*}\label{Formal_properties_of_beta}
In $A_{BN}$ the product $x_{\circ}x_{\bullet}$ is trivial, while $x_{\circ}^{2} = U x_{\circ}$ and $x_{\bullet}^{2} = - U x_{\bullet}$. Furthermore, we have $\Delta_{BN}(x_{\circ}) = x_{\circ} \otimes x_{\circ}$ and $\Delta_{BN}(x_{\bullet}) = x_{\bullet} \otimes x_{\bullet}$.
\end{rem*}

\begin{definition}
Let $D$ be an oriented link diagram. The \emph{$\beta$-cycles} (associated to $D$ over the ring $R$) are homogeneous cycles in $C^{\bullet,\bullet}_{BN}(D,R[U])$, defined as follows
\[\beta(D,R) = \beta(o_{D},R)\qquad \overline{\beta}(D,R) = \beta(-o_{D},R),\]
where $o_{D}$ is the orientation of $D$.
\end{definition}

\begin{proposition}
If $D$ is an oriented link diagram, then $\beta$-cycles associated to $D$ are cycles. Moreover, their homological degree is $0$ and their quantum degree is given by Equation \eqref{quantum_deg_beta}.
\end{proposition}
\begin{proof}
That the $\beta$-cycles are cycles follows from two facts: the set of circles in the oriented resolution is bipartite (see \cite[Lemma 2.4 \& Corollary 2.5]{Rasmussen10}), and $x_{\circ} x_{\bullet} = 0$ (see Remark \ref{Formal_properties_of_beta}). The computations of the homological and quantum degrees are immediate from the definitions.
\end{proof}
The most interesting properties of the $\beta$-cycles can be summarized in the following proposition.

\begin{proposition}[\cite{TransFromKhType17}]\label{Prop:betainvariants}
Let $L$ be an oriented link, and $D$ an oriented diagram representing $L$. Then the following hold
\begin{enumerate}
\item the projection of both  $\beta(D,R)$ and $\overline{\beta}(D,R)$ in Khovanov chain complex is the cycle $\psi = \psi (D)$ defined by Plamenevskaya in \cite{Plamenevskaya06};
\item the homology classes of $\beta(D,R)$ and $\overline{\beta}(D,R)$ are ($\mathbb{F}[U]$-)linearly independent, and non-torsion, in $H_{BN}^{\bullet,\bullet}(L,\mathbb{F}[U])$;
\item $\beta(D,R)$ and $\overline{\beta}(D,R)$ are invariant (up to sign) w.r.t. the maps induced by braid-like Reidemeister moves and positive first Reidemeister moves (see \cite{Khovanov00, BarNatan05cob, Thesis1});
\item if $D^\prime$ is obtained from $D$ by adding an negative kink, then
\[ [\beta(D^\prime,R)] = \pm U (\Phi_{1}^{-})_{*}([\beta(D,R)])\ \ \text{and}\ \ [\overline{\beta}(D^\prime,R)] = \pm U (\Phi_{1}^{-})_{*}([\overline{\beta}(D,R)]),\]
where $\Phi_{1}^{-}$ is the maps associated to the negative first Reidemeister move (see \cite{Khovanov00, BarNatan05cob, Thesis1});
\item if $D$ is the closure of a braid with a row of negative crossings (i.e. exists an index $i$ such that $\sigma_{i}^{-1}$ appears in the word associated to the braid, but $\sigma_{i}$ does not), then
\[ \exists\: x,\: y \in C^{0,\bullet}(D,R[U]) \quad \text{such that}\quad [\beta(D,R)] = U [x],\: [\overline{\beta}(D,R)] = U [y].\]
\end{enumerate}
\qed
\end{proposition}

Using the long exact sequence in homology, induced by the following short exact sequence of (bi-graded) $R$-complexes
\[ 0 \longrightarrow UC_{BN}^{\bullet,\bullet}(D,R[U]) \longrightarrow C_{BN}^{\bullet,\bullet}(D,R[U]) \longrightarrow C_{Kh}^{\bullet,\bullet}(D,R) \to 0\]
together with point (1) of Proposition \ref{Prop:betainvariants}, we obtain that the homology class of $\psi$ is trivial if, and only if, both the homology classes of $\beta(D,R)$ and $\overline{\beta}(D,R)$ are multiples of $U$. Thus, points (4) and (5) of Proposition \ref{Prop:betainvariants} can be interpreted as conditions for the vanishing of $[\psi]$ (cf. \cite[Theorem 3 \& Proposition 3]{Plamenevskaya06}).

Since the maps induced by the Reidemeister moves in Bar-Natan homology are isomorphisms of bi-graded $R[U]$-modules, the integers
\[ c_{R}(D) = max \left\lbrace k\: \vert \: [\beta(D,R)] = U^k x,\: \text{for some }x\in H_{BN}^{0,\bullet}(D, R[U]) \right\rbrace\]
and
\[ \overline{c}_{R}(D) = max \left\lbrace k\: \vert \: [\overline{\beta}(D,R)] = U^k x,\: \text{for some }x\in H_{BN}^{0,\bullet}(D, R[U]) \right\rbrace\]
are transverse braid invariants. In light of what we said above, we can interpret these integers (called \emph{$c$-invariants}) as  the ``degree of vanishing of $[\psi]$''. 
\section{Combinatorial bounds on the $c$- and $s$-invariants}

\subsection{A Bennequin-type inequality}\label{sec:Bennequinineq}

Finding upper bounds for the classical invariants of transverse and Legendrian links in terms of invariants of the topological link type originated as a problem in contact topology, but now has become more a subject of knot and braid theories. These bounds are given by numerical three dimensional (e.g.\: the Seifert genus), four dimensional (e.g.\: the slice genus) or combinatorial link invariants (e.g.\: the un-knotting number), or also by quantities related to polynomial link invariants (such as the Kauffman and HOMFLY-PT polynomials) and, more recently, link homologies (such as Khovanov, Khovanov-Rozansky and knot Floer homologies). This collection of bounds has been named collectively \emph{Bennequin-type inequalities}.

Among the earlier Bennequin-type inequalities there are: the inequality proven by Rudolph (\cite{Rudolph93}), in terms of the slice genus, and the inequality proved by Eliashberg (\cite{Eliashberg92}), in terms of the Euler characteristic of a surface bounding a transverse or a Legendrian link in a tight contact manifold. Then, more than ten years later, Plamenevskaya (\cite{Plamenevskaya06}), and Shumakovitch (\cite{Shumakovitch07}), introduced similar inequalities for the featuring the $s$-invariant instead of the genus (we call them \emph{Bennequin $s$-inequalities}).

In the next proposition we will prove a new Bennequin $s$-inequality, featuring the $c$-invariants, which sharpens the bound given by Plamenevskaya and Shumakovitch. This inequality will be a key ingredient to produce our bound.

\begin{rem*}
We will state the next proposition in terms of the invariant $c_{R}$. The statement holds true also by replacing $c_{R}$ with $\overline{c}_{R}$.
\end{rem*}

\begin{proposition}[Bennequin $s$-inequality]\label{Prop:Bennequins}
Let $D$ be an oriented link diagram representing the link-type $\lambda$, and let $\mathbb{F}$ be a field, then
\[ s(\lambda;\mathbb{F}) \geq w(D) - V(D) + 2 c_{\mathbb{F}}(D) + 1,\]
where $w$ is the writhe of $L$ and $V(D)$ is the number of circles in the oriented resolution of $D$.
\end{proposition}
\begin{proof}
In \cite[Subsection 7.1]{BeliakovaWehrli08} (see also \cite[Proposition 3.3 \& Corollary 3.6]{Rasmussen10}) it has been observed that
\[ s(D;\mathbb{F}) - 1 = min\left\lbrace  Fdeg([\mathbf{v}_{o_{D}}(D) - \mathbf{v}_{-o_{D}}(D)]), Fdeg([\mathbf{v}_{o_{D}}(D) + \mathbf{v}_{-o_{D}}(D)])\right\rbrace,\]
where $o_{D}$ is the orientation of $D$.
 Since, both $\{ [\mathbf{v}_{o_{D}}(D) - \mathbf{v}_{-o_{D}}(D)], [\mathbf{v}_{o_{D}}(D) + \mathbf{v}_{-o_{D}}(D)]\}$, and $\{ [\mathbf{v}_{o_{D}}(D)], [\mathbf{v}_{-o_{D}}(D)]\}$, generate the same sub-module of $H_{TLee}^{\bullet}(D,\mathbb{F})$, it is easy to see that the minimum filtered degree of the elements of these two sets is the same (\cite[Corollary A.6]{Thesis1}).  The automorphism (as filtered vector space) of $A_{TLee}$ sending $1$ to $1$ and $x_{\circ}$ to $x_{\bullet}$, induces an automorphism of $C_{TLee}^\bullet$ which commutes (up to sign) with the differential and preserves the filtration. Since this automorphism exchanges  $[\mathbf{v}_{o_{D}}(D)]$ and $[\mathbf{v}_{-o_{D}}(D)]$, they have the same filtered degree (cf. with the proof of \cite[Corollary 3.6]{Rasmussen10}).

Thus, we obtained that
\[s(D;\mathbb{F}) - 1 = Fdeg([\mathbf{v}_{o_{D}}(D)]) = max \{ Fdeg(x)\: \vert\: x\in [\mathbf{v}_{o_{D}}(D)]\},\]
where the filtered degree $Fdeg(x)$ is defined as the maximal $j$ such that $x\in \mathscr{F}_{j}$.
Since
\[[\mathbf{v}_{o_{D}}(D)]=(\pi_{TLee})_{*}([\beta(D,\mathbb{F})]) = (\pi_{TLee})_{*}(U^{c_{\mathbb{F}}(D)}[x]) = (\pi_{TLee})_{*}([x]) = [\pi_{TLee}(x)],\]
for some $x$ such that $qdeg(x) = w(D) -V(D) + 2 c_{\mathbb{F}}(D)$, and since $ qdeg(x) \leq Fdeg(\pi_{TLee}(x))$ (cf. Remark \ref{gradingsFA}), and the proposition follows. 
\end{proof}

\begin{rem*}
If one replaces $[\mathbf{v}_{o_{D}}(D)]$ with $[\mathbf{v}_{-o_{D}}(D)]$ in the proof of Proposition \ref{Prop:Bennequins}, it also has to replace $\beta$ with $\overline{\beta}$, and $c$ with $\overline{c}$, obtaining the analogue statement for $\overline{c}$.
\end{rem*}

As a corollary we get Theorem  \ref{thm:main2}.

\begin{proof}[Proof of Theorem \ref{thm:main2}]
It is an immediate consequence of Proposition \ref{Prop:Bennequins} (just choose $D$ to be a closed braid), and from the definition of self-linking number in terms of closed braid diagrams.
\end{proof}
\subsection{Combinatorial bounds}\label{Combinatorialbounds}

Now we are ready to provide an lower bound for the value of the $c$-invariants. As in the case of Proposition \ref{Prop:Bennequins}, we will prove everything using $c$, but all the results hold by replacing $c$ with $\overline{c}$.
\begin{warning}
In the follow up we are not going to make any distinction between a vertex of the (simplified) Seifert graph, and the corresponding circle in the oriented resolution. Moreover, given two resolutions $\underline{r}$ and $\underline{s}$ which differ only for a local resolution at a crossing, say $c$, we will identify the circles of $\underline{r}$ which do not touch $c$ with the corresponding circles in $\underline{s}$.
\end{warning}

Let $D$ be an oriented link diagram.  Denote by $\Gamma_- = \Gamma_- (D)$ the full sub-graph of $\Gamma(D)$ spanned by the negative vertices.
Given $\Gamma^{\prime}$ a full sub-graph of $\Gamma_-$, define $x(\Gamma^{\prime})$ to be the enhanced state in $C_{BN}^{\bullet,\bullet}(D,R)$ whose underlying resolution is the oriented resolution, obtained by labeling a circle $v$ as follows
\[ x_v = \begin{cases} 1_{BN} & \text{if}\: v\in \Gamma_{-}\setminus \Gamma^\prime\\ \text{the same label as in }\beta(D,R) & \text{otherwise} \end{cases}\]
Now fix a vertex of $\Gamma^\prime$, say $v_{0}$, and define $x(\Gamma^{\prime}, v_{0})$ to be the enhanced cycle which is identical to $x(\Gamma^{\prime})$ except on $v_{0}$, where the label is the conjugate of the corresponding label in $\beta(D,R)$ -- i.e.\: the label is $x_{\circ}$ if the corresponding label in $\beta(D,R)$ is $x_\bullet$, and \emph{vice versa}. Now that all the notation is set, we can prove the following lemma.

\begin{lemma}\label{Lem:combinatoriallemmaboundonc}
Let $D$ be an oriented link diagram. If $\Gamma^{\prime}$ is a (non-empty) full sub-graph of $\Gamma_{-}(D)$, and $v_{0}\in V(\Gamma^{\prime})$, then
\begin{enumerate}
\item $\pm U x(\Gamma^{\prime}\setminus \{ v \}) = x(\Gamma^{\prime}) - x(\Gamma^{\prime},v_{0})$;
\item if $v_{0}\in V(\Gamma^{\prime})$ is a non-pure or a non-isolated (in $\Gamma^\prime$) vertex, then $x(\Gamma^{\prime},v_{0})$ is a boundary. 
\end{enumerate}
\end{lemma}
\begin{proof}
Point (1) follows directly from the definitions. Thus, we can concentrate on (2). Since $v_{0}$ is either non-pure or non-isolated, there is at least a vertex, say $v$, which is not in $\Gamma_{-}\setminus \Gamma^\prime$ and is connected to $v_{0}$ by a negative edge.
\begin{figure}[h]
\centering
\begin{tikzpicture}[scale = .6]

\draw[gray] (2,5) .. controls +(0,-.5) and +(0,.5) .. (0,4);
\draw[gray] (2,3.5) .. controls +(0,-.5) and +(0,.5) .. (0,2.5);
\draw[gray] (2,2) .. controls +(0,-.5) and +(0,.5) .. (0,1);
\draw[gray] (2,0) .. controls +(0,-.5) and +(0,.5) .. (0,-1);

\pgfsetlinewidth{8*\pgflinewidth}
\draw[white] (0,5) .. controls +(0,-.5) and +(0,.5) .. (2,4);
\draw[white]  (0,3.5) .. controls +(0,-.5) and +(0,.5) .. (2,2.5);
\draw[white]  (0,2) .. controls +(0,-.5) and +(0,.5) .. (2,1);
\draw[white]  (0,0) .. controls +(0,-.5) and +(0,.5) .. (2,-1);
\pgfsetlinewidth{.125*\pgflinewidth}
\draw[gray] (0,5) .. controls +(0,-.5) and +(0,.5) .. (2,4);
\draw[gray] (0,3.5) .. controls +(0,-.5) and +(0,.5) .. (2,2.5);
\draw[gray] (0,2) .. controls +(0,-.5) and +(0,.5) .. (2,1);
\draw[gray] (0,0) .. controls +(0,-.5) and +(0,.5) .. (2,-1);

\node at (1,.5) {$\vdots$};
\node at (-.5,3.5) {$v_0$};
\node at (2.5,3.5) {$v$};
\node at (1,1.25) {$c$};
\node at (1,-1.5) {$\underline{r}$};


\draw[dashed] (-1,5) -- (3,5) -- (3,-1) -- (-1,-1) -- cycle;

\draw[thick] (0,5) .. controls +(.25,-.5) and +(.25,.5) .. (0,4);
\draw[thick] (0,4) -- (0,3.5);
\draw[thick] (0,3.5) .. controls +(.25,-.5) and +(.25,.5) .. (0,2.5);
\draw[thick,->] (0,2.5) -- (0,2);
\draw[thick] (0,2) .. controls +(.25,-.5) and +(.25,.5) .. (0,1);
\draw[thick] (0,1) -- (0,0);
\draw[thick] (0,0) .. controls +(.25,-.5) and +(.25,.5) .. (0,-1);

\draw[thick] (2,0) .. controls +(-.25,-.5) and +(-.25,.5) .. (2,-1);
\draw[thick] (2,5) .. controls +(-.25,-.5) and +(-.25,.5) .. (2,4);
\draw[thick] (2,4) -- (2,3.5);
\draw[thick] (2,3.5) .. controls +(-.25,-.5) and +(-.25,.5) .. (2,2.5);
\draw[thick,->] (2,2.5) -- (2,2);
\draw[thick] (2,2) .. controls +(-.25,-.5) and +(-.25,.5) .. (2,1);
\draw[thick] (2,1) -- (2,0);

\begin{scope}[shift = {(5,0)}]

\draw[gray] (2,5) .. controls +(0,-.5) and +(0,.5) .. (0,4);
\draw[gray] (2,3.5) .. controls +(0,-.5) and +(0,.5) .. (0,2.5);
\draw[gray] (2,2) .. controls +(0,-.5) and +(0,.5) .. (0,1);
\draw[gray] (2,0) .. controls +(0,-.5) and +(0,.5) .. (0,-1);

\pgfsetlinewidth{8*\pgflinewidth}
\draw[white] (0,5) .. controls +(0,-.5) and +(0,.5) .. (2,4);
\draw[white]  (0,3.5) .. controls +(0,-.5) and +(0,.5) .. (2,2.5);
\draw[white]  (0,2) .. controls +(0,-.5) and +(0,.5) .. (2,1);
\draw[white]  (0,0) .. controls +(0,-.5) and +(0,.5) .. (2,-1);
\pgfsetlinewidth{.125*\pgflinewidth}
\draw[gray] (0,5) .. controls +(0,-.5) and +(0,.5) .. (2,4);
\draw[gray] (0,3.5) .. controls +(0,-.5) and +(0,.5) .. (2,2.5);
\draw[gray] (0,2) .. controls +(0,-.5) and +(0,.5) .. (2,1);
\draw[gray] (0,0) .. controls +(0,-.5) and +(0,.5) .. (2,-1);

\node at (1,.5) {$\vdots$};
\node at (-.5,3.5) {$\gamma$};
\node at (1,1) {$c$};
\node at (1,-1.5) {$\underline{s}$};


\draw[dashed] (-1,5) -- (3,5) -- (3,-1) -- (-1,-1) -- cycle;

\draw[thick] (0,5) .. controls +(.25,-.5) and +(.25,.5) .. (0,4);
\draw[thick] (0,4) -- (0,3.5);
\draw[thick] (0,3.5) .. controls +(.25,-.5) and +(.25,.5) .. (0,2.5);
\draw[thick] (0,2.5) -- (0,2);
\draw[thick] (0,1) .. controls +(.25,.5) and +(-.25,.5) .. (2,1);
\draw[thick] (0,1) -- (0,0);
\draw[thick] (0,0) .. controls +(.25,-.5) and +(.25,.5) .. (0,-1);

\draw[thick] (2,0) .. controls +(-.25,-.5) and +(-.25,.5) .. (2,-1);
\draw[thick] (2,5) .. controls +(-.25,-.5) and +(-.25,.5) .. (2,4);
\draw[thick] (2,4) -- (2,3.5);
\draw[thick] (2,3.5) .. controls +(-.25,-.5) and +(-.25,.5) .. (2,2.5);
\draw[thick] (2,2.5) -- (2,2);
\draw[thick] (2,2) .. controls +(-.25,-.5) and +(.25,-.5) .. (0,2);
\draw[thick] (2,1) -- (2,0);
\end{scope}
\end{tikzpicture}
\caption{The circles $v_0$ and $v$ in the oriented resolution $\underline{r}$, the crossing $c$ and the circle $\gamma$ in the resolution $\underline{s} = \underline{s}(c)$.}
\label{fig:proofofdeltaelle}
\end{figure}
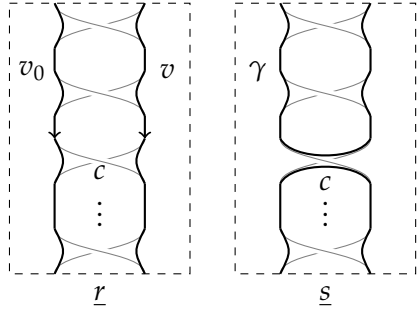

Choose a crossing $c$ connecting $v_{0}$ and $v$. Denote by $\underline{s}$ the resolution which differs from the oriented resolution of $L$, denoted by $\underline{r}$, only in $c$. Notice that since $c$ is a negative crossing in $L$, it follows that $\vert \underline{s} \vert < \vert \underline{r} \vert$. Denote by $\gamma$ the circle in $\underline{s}$ obtained by merging $v$ and $v_{0}$ (see Figure \ref{fig:proofofdeltaelle}). Now, define $y(\Gamma^\prime,v_{0},c)$ to be the enhanced state whose underlying resolution is $\underline{s}$, where all the circles, other than $\gamma$, have the same label as in $x(\Gamma^\prime,v_{0})$, and $\gamma$ has label $x_{v}$. We claim that
\[d_{BN}(y(\Gamma^\prime,v_{0},c)) = x(\Gamma^\prime,v_{0}).\]
Since $\underline{s}$ differs from $\underline{r}$ only in $c$, all the crossing which are replaced by their $0$-resolutions are either positive crossings, or $c$. 

In order to split a circle by changing a single local resolution, one has to find a crossing joining the circle with itself. Since by changing a resolution in the oriented resolution we merge two circles, the only circle in $\underline{s}$ which has such crossings is $\gamma$. But all the crossings touching only $\gamma$ are those which used to connect $v$ and $v_{0}$, and thus are all resolved with $1$-resolutions except from $c$. Thus all the contributions to $d_{BN}(y(\Gamma^\prime,v_{0},c))$ are obtained by merging two circles along a positive crossing (type A), or splitting $\gamma$ along $c$ (type B). All contributions of type A are trivial; these are (by definition) given by the enhanced states whose labels are the same as in $y(\Gamma^\prime,v_{0},c)$ except in the circle, say $\gamma^\prime$, obtained by merging two circles, say $\gamma_{0}$ and $\gamma_{1}$. The label of $\gamma^\prime$ is obtained by multiplying the labels which $\gamma_{0}$ and $\gamma_{1}$ had in $y(\Gamma^\prime,v_{0},c)$. However, these labels are the same as in $\beta(L,R)$, and thus they are conjugate. It follows that $\gamma^\prime$ has label $0$. 
So, $d_{BN}(y(\Gamma^\prime,v_{0},c))$ consist of a single enhanced state, coming from the single contribution of type B. Thanks to the behaviour of $x_{\circ}$ and $x_{\bullet}$ with respect to the co-multiplication (cf. Remark \ref{Formal_properties_of_beta}), the claim follows easily from the definition of $d_{BN}$.
\end{proof}

\begin{lemma}\label{Lem:deltaminus}
Let $D$ be an oriented link diagram, and let $\Gamma^\prime$ be any full sub-graph of $\Gamma_{-}(D)$. If there is a negative edge between two neutral vertices in $\Gamma(D)$, then 
\[ x( \Gamma^\prime ) = U y + d_{BN} z,\]
for some $y,\: z\in C_{BN}^{\bullet,\bullet}(D,R[U])$.
\end{lemma}
\begin{proof}
Denote by $v^\prime$ and $v^{\prime\prime}$ any pair of neutral vertices connected by a negative edge. Choose a crossing $c$ connecting $v^\prime$ and $v^{\prime\prime}$. Let $\underline{s}$ the resolution obtained from the oriented resolution $\underline{r}$ by replacing the $1$-resolution in $c$ with its $0$-resolution. Define the enhanced cycle $z(\Gamma^\prime,c)$ as follows
\[z_{\gamma} = \begin{cases} x_{v^\prime} & \text{if}\: \gamma\text{ is the circle}\:\gamma ^\prime\: \text{obtained by merging }v^\prime\text{ and }v^{\prime\prime}\\
x_v &\text{if the circle }\gamma\text{ can be identified with the circle }v\in\underline{r}\end{cases}\]
\begin{rem*}
The whole argument applies without change if we replace $v^\prime$ with $v^{\prime\prime}$ in the definition of $z$.
\end{rem*}
Let us compute $d_{BN}z(\Gamma^\prime,c)$; we claim that
\[d_{BN}z(\Gamma^\prime,c) =  x( \Gamma^\prime ) - U y.\]
There are three kinds of enhanced states which contribute to $d_{BN}z(\Gamma^\prime,c)$: those whose resolution is obtained by changing the resolution at a positive crossing touching $\gamma^\prime$ (type A), those whose resolution is obtained by changing the resolution at a positive \emph{not} touching $\gamma$ (type B), and the enhanced state whose resolution is $\underline{r}$ (type C).

Type A contribution are obtained by merging positive or neutral circles with $\gamma^\prime$. The labels of these circles are exactly the same as in $\beta(L;R)$; in particular, they are either conjugate or equal to $z_{\gamma^\prime}$. In the first case, the contribution is trivial, in the second case the enhanced states obtained are multiples of $U$ (cf. Remark \ref{Formal_properties_of_beta}). Notice that there is at least one contribution of type A coming from the latter type of circles (otherwise, $v^\prime$ would be a negative circle). It is immediate that type B contribution are trivial. Finally, it is easy to check that the single type C contribution is exactly $x(\Gamma^\prime)$, and this proves the claim.
\end{proof}

Now, we can prove the following theorem.

\begin{theorem}\label{Thm:estimateonc}
Let $D$ be an oriented link diagram representing an oriented link $L$, and $R$ be a ring. Denote by $\ell_-$ be the number of split components of $D$ (i.e.\: connected components of $D$ seen as a four-valent graph) which have only negative crossings. Then
\[ \beta(D) = U^{V_-(D) - \ell_- + \delta_{-}(D)} y + d_{BN}z, \]
for some $y,\: z \in C_{BN}^{\bullet,\bullet}(D,R[U])$.
\end{theorem}
\begin{proof}
We will prove the theorem by a repeated use of Lemma \ref{Lem:combinatoriallemmaboundonc}. More precisely, we will define a sequence of vertices, say $v_{i} \in \Gamma_-$, such that if we set (recursively)
\[ \Gamma_{i} = \Gamma_{i-1} \setminus \{ v_{i}\}\quad\text{and}\quad \Gamma_{0} = \Gamma_{-}, \]
the vertex $v_{i}$ is either non-pure or non-isolated in $\Gamma_{i-1}$; moreover, we wish the remaining vertices, say $\{ v^\prime_{1},...,v^\prime_{\ell_{-}}\}$ to lay in pairwise disjoint connected components of $\Gamma_{-}$. In this way the set $\{ v^\prime_{1},...,v^\prime_{\ell_{-}}\}$ is a full sub-graph of $\Gamma_{-}$.

Suppose that we have the sequence of vertices $\{ v_i \}_{i}$ as above. By applying (1) of Lemma \ref{Lem:combinatoriallemmaboundonc} several times, we obtain
\[ \beta(L,R) = x(\Gamma_0) = \pm U x(\Gamma_{1}) + x(\Gamma_{0},v_{1}) = \dots =\] 
\[ = \pm U^{V_-(D) - \ell_-}x( \{ v^\prime_{1},...,v^\prime_{\ell_{-}}\})  + \sum_{i=1}^{V_-(D) - \ell_-} \pm U^{i} x(\Gamma_{i-1},v_{i}).\]
The last summand is a boundary thanks to (2) of Lemma \ref{Lem:combinatoriallemmaboundonc}. Moreover, if there is a negative edge between two neutral vertices in $\Gamma(D)$, we can apply Lemma \ref{Lem:deltaminus} with $\Gamma^\prime = \{ v^\prime_{1},...,v^\prime_{\ell_{-}}\}$, and the claim follows.

To define the sequence of vertices $v_{i} \in V(\Gamma_{-}) \setminus \{  v^\prime_{1},...,v^\prime_{\ell_-} \}$, one can start by fixing an order of the connected components of $\Gamma_{-}$, say $\Gamma_1 ,...,\Gamma_k$.
For each component consider a spanning tree. There is way to define a total order on each spanning tree once one fixes a root: first consider the distance from the root, and then choose an arbitrary order on the nodes which are at the same distance from the root.
The choice of the root can be almost arbitrary: one just have to pay attention to avoid pure vertices whenever possible.
The only spanning trees having all pure vertices are those spanning connected components of $\Gamma_-$ which are also connected components of $\Gamma$. These connected components of $\Gamma$ are exactly those corresponding to negative split components of $D$. 
Now, fix a root on each spanning tree and an order on its vertices, as described above.

Denote by $ v^\prime_{1},...,v^\prime_{\ell_-} $ the roots of the spanning trees spanning connected components of $\Gamma_-$ which are also connected components of $\Gamma$. 
Order all the vertices of $\Gamma_{-}$, except $ v^\prime_{1},...,v^\prime_{\ell_-} $, in the following way: first look at the order of the connected components to which they belong. Then, if they belong to the same component, use the order on the spanning tree. 
Finally, define $v_{i}$ to be the $i$-th vertex with respect to this total order. Notice that each $v_{i}$ is (by definition) either non-pure or non-isolated in $\Gamma_{i-1}$, and this concludes the proof.
\end{proof}

The following corollary is immediate.

\begin{corollary}\label{Cor:combinatorialboundonc}
Let $D$ be an oriented link diagram, then
\[c_{R}(D) \geq V_{-}(D) - \ell_{-}(D) + \delta_{-}(D)\]
\end{corollary}
\begin{proof}
The result follows immediately from Theorem \ref{Thm:estimateonc} and from the definition of $c$-invariants.
\end{proof}

\begin{proof}[Proof of Theorem \ref{thm:main1}]
Just apply Corollary \ref{Cor:combinatorialboundonc} to closed braid diagrams.
\end{proof}

\begin{proof}[Proof of Theorem \ref{thm:main3}]
The combinatorial bound is immediate from Corollary \ref{Cor:combinatorialboundonc} and Proposition \ref{Prop:Bennequins}.
The sharpness result is application of the upper bound in \eqref{Lobbbound}; it is sufficient to notice that if the diagram $D$ is negative or positive then $\delta_{-}(D) = 0$, and
\[o_{-}(D) - \ell_{-}(D) = - s_{-}(D) + V(D).\]
\end{proof}

\begin{rem*}
The bounds \eqref{Lobbbound} and \eqref{KawCavbound} are sharp in the case of homogeneous knot diagrams, as proved by T. Abe (\cite{Abe11}), which include the case of positive and negative diagrams. However, this is not true in general for \eqref{combinatorialboundons}, see Remark \ref{Rem:homogeneouscounterexample}.
\end{rem*}

\subsection{Comparisons and examples}\label{sec:examples}
In this subsection we show that the bound in Equation \eqref{combinatorialboundons} is independent from the lower bound in \eqref{Lobbbound}, and from the bound in \eqref{KawCavbound}; more precisely, we produce few family of examples where one of the above mentioned bounds is sharp, while the others are not. Before going on, we wish to point out two things: (1) all these bounds are strongly dependent on the diagram (\cite[Chapter 4, Subsection 2.3]{Thesis1}), and (2) since the lower bound in \eqref{Lobbbound} is strictly weaker than the bound in \eqref{KawCavbound}, it is sufficient to find an example where \eqref{combinatorialboundons} is sharper than \eqref{KawCavbound}.

\subsubsection*{An interesting unknot}

It is not difficult to provide an example of knot diagram where \eqref{KawCavbound} (and hence\footnote{Notice that \eqref{KawCavbound} and \eqref{Lobbbound} coincide in the case of knots.} \eqref{Lobbbound}) provide a better bound than \eqref{combinatorialboundons}. Such an example is given by the diagram $U(k,h)$ of the unknot depicted in Figure \ref{fig:interesting_unknot}.
\begin{figure}[h]
\centering
\begin{tikzpicture}[scale = .4, thick]

\draw[] (-2,4.25) .. controls +(0,.5) and +(-.5,0) .. (-1,5.25);
\draw[] (2,4.25) .. controls +(0,.5) and +(.5,0) .. (1,5.25);

\draw[->] (-1,0) .. controls +(1,1) and +(-1,0) .. (1,5.25);

\draw[] (-3,0) .. controls +(0,-1) and +(-1,-1) .. (-1,0);
\draw[<-] (3,0) .. controls +(0,-1) and +(1,-1) .. (1,0);
\pgfsetlinewidth{12.5*\pgflinewidth}
\draw[white] (-2,0) .. controls +(1,-2) and +(-1,-2) .. (2,0);
\draw[white] (1,0) .. controls +(-1,1) and +(1,0) .. (-1,5.25);
\pgfsetlinewidth{.08*\pgflinewidth}
\draw[] (-2,0) .. controls +(1,-2) and +(-1,-2) .. (2,0);
\draw[] (-3,4.25) .. controls +(0,3) and +(0,3) .. (3,4.25);
\draw[] (1,0) .. controls +(-1,1) and +(1,0) .. (-1,5.25);

\draw (-3.25,0) rectangle (-1.75,1.5);
\draw (-3,1.5) -- (-3,1.75) (-2,1.5) -- (-2,1.75) ;
\draw (-3.25,2.75) rectangle (-1.75,4.25);
\draw (-3,2.5) -- (-3,2.75) (-2,2.5) -- (-2,2.75) ;

\draw (3.25,0) rectangle (1.75,1.5);
\draw (3,1.5) -- (3,1.75) (2,1.5) -- (2,1.75) ;
\draw (3.25,2.75) rectangle (1.75,4.25);
\draw (3,2.5) -- (3,2.75) (2,2.5) -- (2,2.75) ;

\node at (-2.5,3.5) {$T$};
\node at (2.5,3.5) {$T$};
\node at (-2.5,.75) {$T$};
\node at (2.5,.75) {$T$};
\node at (2.5,2.25) {$\vdots$};
\node at (-2.5,2.25) {$\vdots$};

\node at (0,-3) {$U(k,h)$};
\node at (4.5,2.25) {$h$};
\node at (-4.5,2.25) {$k$};
\draw[fill] (-3.5,0) .. controls +(-.5,0) and +(.5,0) .. (-4,2.25) .. controls +(.625,0) and +(-.425,0) .. (-3.5,0);
\draw[fill] (-3.5,4.25) .. controls +(-.5,0) and +(.5,0) .. (-4,2.25) .. controls +(.625,0) and +(-.425,0) .. (-3.5,4.25);
\draw[fill] (3.5,0) .. controls +(.5,0) and +(-.5,0) .. (4,2.25) .. controls +(-.625,0) and +(.425,0) .. (3.5,0);
\draw[fill] (3.5,4.25).. controls +(.5,0) and +(-.5,0) .. (4,2.25) .. controls +(-.625,0) and +(.425,0) ..  (3.5,4.25);
\begin{scope}[shift={+(0,-.75)}]
\node at (6.5,3) {$T = $};
\draw[dashed] (8,0) rectangle (12,6);
\draw[->] (11.5,-.5) .. controls +(0,2) and +(0,-1) .. (8.5,3) .. controls +(0,1) and +(0,-2) .. (11.5,6.5);
\pgfsetlinewidth{12.5*\pgflinewidth}
\draw[white]  (8.5,-.5) .. controls +(0,2) and +(0,-1) .. (11.5,3) .. controls +(0,1) and +(0,-2) .. (8.5,6.5);
\pgfsetlinewidth{.08*\pgflinewidth}
\draw[<-] (8.5,-.5) .. controls +(0,2) and +(0,-1) .. (11.5,3) .. controls +(0,1) and +(0,-2) .. (8.5,6.5);
\end{scope}

\begin{scope}[shift={+(19,0)}]
\node at (-1,-3) {$\Gamma\left(U(k,h)\right)$};
\node at (3,2.5) {$2h $};
\node at (-5,2.5) {$2k $};
\draw[fill] (-3.5,0) .. controls +(-.5,0) and +(.5,0) .. (-4,2.5) .. controls +(.625,0) and +(-.425,0) .. (-3.5,0);
\draw[fill] (-3.5,5) .. controls +(-.5,0) and +(.5,0) .. (-4,2.5) .. controls +(.625,0) and +(-.425,0) .. (-3.5,5);
\draw[fill] (1.5,0) .. controls +(.5,0) and +(-.5,0) .. (2,2.5) .. controls +(-.625,0) and +(.425,0) .. (1.5,0);
\draw[fill] (1.5,5).. controls +(.5,0) and +(-.5,0) .. (2,2.5) .. controls +(-.625,0) and +(.425,0) ..  (1.5,5);
\end{scope}
\draw[red, thick] (18,-1) -- (20,0) (18,6) -- (20,5) (18,6) -- (18,-1) (20,1.5) -- (20,2)  (20,3.5) -- (20,3) (16,5)--(16,3.5) (16,0)--(16,1.5) ;

\draw[blue, thick] (18,-1) -- (16,0) (18,6) -- (16,5)  (16,1.5) -- (16,2)  (16,3.5) -- (16,3) (20,5)--(20,3.5) (20,0)--(20,1.5) ;

\draw[green,fill] (18,6) circle [radius=.1] ;
\draw[ white, fill] (18,6) circle [radius=.02] ;

\draw[green,fill] (18,-1) circle [radius=.1] ;
\draw[ white, fill] (18,-1) circle [radius=.02] ;

\draw[green,fill] (20,0) circle [radius=.1] ;
\draw[ white, fill] (20,0) circle [radius=.02] ;

\draw[green,fill] (16,3.5) circle [radius=.1] ;
\draw[ white, fill] (16,3.5) circle [radius=.02] ;

\draw[green,fill] (20,3.5) circle [radius=.1] ;
\draw[ white, fill] (20,3.5) circle [radius=.02] ;

\draw[green,fill] (16,1.5) circle [radius=.1] ;
\draw[ white, fill] (16,1.5) circle [radius=.02] ;

\draw[green,fill] (20,1.5) circle [radius=.1] ;
\draw[ white, fill] (20,1.5) circle [radius=.02] ;

\draw[green,fill] (16,0) circle [radius=.1] ;
\draw[ white, fill] (16,0) circle [radius=.02] ;
\draw[green,fill] (20,5) circle [radius=.1] ;
\draw[ white, fill] (20,5) circle [radius=.02] ;

\draw[green,fill] (16,5) circle [radius=.1] ;
\draw[ white, fill] (16,5) circle [radius=.02] ;
\end{tikzpicture}
\caption{The diagram $(k,h)$, the tangle $T$, and the simplified Seifert graph associated to $U(h,k)$. }
\label{fig:interesting_unknot}
\end{figure}
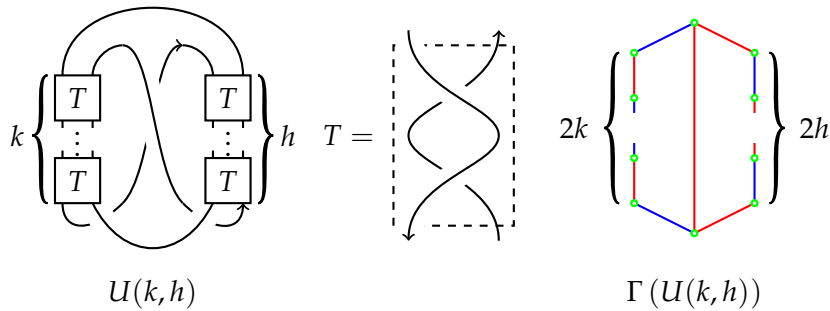

In Table \ref{tab:interesting_unknot} we report all the quantities involved in the computations of the bounds in \eqref{Lobbbound}, \eqref{KawCavbound} and \eqref{combinatorialboundons}, as well as the values of the bounds. 

\begin{table}[H]
\centering
\begin{tabular}{ccc}
Quantities & Value\\
\hline\hline
$w_{\ }$  & $1$   \\
$V_{\ }$  & $2(k+h+1)$ \\
$V_-$     & $0$ \\
$s_+$     & $k+h$ \\
$s_-$     & $k+h+1$ \\
$\ell_s$  & $1$ \\
$\ell_{\ }$  & $1$ \\
$\ell_-$  & $0$ \\
$\delta_-$& $1$ \\\hline
\eqref{Lobbbound} upper & $ 2$\\
\eqref{combinatorialboundons}& $ -2(k+h-1)$\\
\eqref{KawCavbound} & $-2$\\
\end{tabular}
\caption{Comparing the bounds for $U(k,h)$.}\label{tab:interesting_unknot}
\end{table}
We can notice that, while the bounds given in \eqref{Lobbbound} and \eqref{KawCavbound} are constant, the bound provided in \eqref{combinatorialboundons} decreases linearly with $k$ and $h$. 

\begin{rem*}\label{Rem:homogeneouscounterexample}
With the same strategy one may concoct other examples of links (or even knots) for which \eqref{KawCavbound} (or also \eqref{Lobbbound}) is sharper than \eqref{combinatorialboundons}; the addition of copies of the tangle $T$ in Figure \ref{fig:interesting_unknot} to a (suitable) oriented diagram $D$ increases the quantity $V(D)$, which comes with a minus sign, while leaves unchanged all the other quantities involved in \eqref{combinatorialboundons}. This decreases the value on the right hand side of \eqref{combinatorialboundons}. On the other hand, the addition of copies of $T$ also increases the quantity $2s_{+}(D)$, which comes with a positive sign, compensating $V(D)$ in the right hand side of \eqref{KawCavbound} and \eqref{Lobbbound}. As an example the reader may compute the r.h.s. of \eqref{combinatorialboundons} and of \eqref{KawCavbound} in the case of the diagram $K_{h}$ in Figure \ref{fig:homogeneousdiagram}. Notice that $K_{h}$ is a homogeneous diagram in the sense of \cite{Abe11}, and the corresponding knot type does not depend on $h$.
\end{rem*}

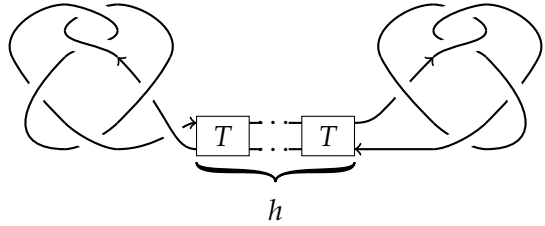
\begin{figure}[h]
\centering
\begin{tikzpicture}[scale=.35, thick]

\draw (0,4) .. controls +(-.5,0) and +(0,.5) .. (-2,3);
\draw (0,3)  .. controls +(0,-.5) and +(.5,.5) .. (-2,2);
\pgfsetlinewidth{10*\pgflinewidth}
\draw[white]  (-2,4) .. controls +(.5,0) and +(0,.5) .. (0,3);
\draw[white] (-2,3)  .. controls +(0,-.5) and +(-.5,.5) .. (0,2);
\pgfsetlinewidth{.1*\pgflinewidth}
\draw (-2,4) .. controls +(.5,0) and +(0,.5) .. (0,3);
\draw (-2,3)  .. controls +(0,-.5) and +(-.5,.5) .. (0,2);

\draw (0,4)  .. controls +(1,0) and +(.5,1) .. (2,2);
\draw (-2,4)  .. controls +(-1,0) and +(-.5,1) .. (-4,2);

\draw (0,-.5) .. controls +(-.5,-.5) and +(.5,0) .. (-2,-1.5);
\pgfsetlinewidth{10*\pgflinewidth}
\draw[white]  (-2,-.5) .. controls +(.5,-.5) and +(-.5,0) .. (0,-1.5);
\pgfsetlinewidth{.1*\pgflinewidth}
\draw (-2,-.5) .. controls +(.5,-.5) and +(-.5,0) .. (0,-1.5);

\draw[->] (0,-1.5)   .. controls +(1.5,0) and +(-.5,0) .. (3,-.5);
\pgfsetlinewidth{10*\pgflinewidth}
\draw[white] (3,-1.5)   .. controls +(-1,0) and +(2,-2) .. (0,2);
\pgfsetlinewidth{.1*\pgflinewidth}
\draw[->] (3,-1.5)   .. controls +(-1,0) and +(2,-2) .. (0,2);
\pgfsetlinewidth{10*\pgflinewidth}
\draw[white] (0,-.5)  .. controls +(.5,.5) and +(-.5,-1) .. (2,2);
\pgfsetlinewidth{.1*\pgflinewidth}
\draw (0,-.5)  .. controls +(.5,.5) and +(-.5,-1) .. (2,2);

\draw (-2,-.5)   .. controls +(-.5,.5) and +(.5,-1) .. (-4,2);
\pgfsetlinewidth{10*\pgflinewidth}
\draw[white] (-2,-1.5)   .. controls +(-2,0) and +(-2,-2) .. (-2,2);
\pgfsetlinewidth{.1*\pgflinewidth}
\draw (-2,-1.5)   .. controls +(-2,0) and +(-2,-2) .. (-2,2);

\draw[thin] (3,-.25) rectangle (5,-1.75);

\draw[thin] (9,-.25) rectangle (7,-1.75);

\draw (5.5,-.5) -- (5,-.5);
\draw (5.5,-1.5) -- (5,-1.5);
\draw (6.5,-.5) -- (7,-.5);
\draw (6.5,-1.5) -- (7,-1.5);

\node at (6,-.5) {$\dots$};
\node at (6,-1.5) {$\dots$};

\node at (4,-1) {$T$};
\node at (8,-1) {$T$};

\draw[fill] (3,-2)   .. controls +(0,-.75) and +(0,.75) .. (6,-2.75) .. controls +(0,.75) and +(0,-.75) .. (9,-2)   .. controls +(0,-.5) and +(0,.75) .. (6,-2.75) .. controls +(0,.75) and +(0,-.5) .. (3,-2);

\node[below] at (6,-3) {$h$};

\begin{scope}[shift={(12,0)}, xscale=-1]
\draw (0,4) .. controls +(-.5,0) and +(0,.5) .. (-2,3);
\draw (0,3)  .. controls +(0,-.5) and +(.5,.5) .. (-2,2);
\pgfsetlinewidth{10*\pgflinewidth}
\draw[white]  (-2,4) .. controls +(.5,0) and +(0,.5) .. (0,3);
\draw[white] (-2,3)  .. controls +(0,-.5) and +(-.5,.5) .. (0,2);
\pgfsetlinewidth{.1*\pgflinewidth}
\draw (-2,4) .. controls +(.5,0) and +(0,.5) .. (0,3);
\draw (-2,3)  .. controls +(0,-.5) and +(-.5,.5) .. (0,2);

\draw (0,4)  .. controls +(1,0) and +(.5,1) .. (2,2);
\draw (-2,4)  .. controls +(-1,0) and +(-.5,1) .. (-4,2);

\draw (0,-.5) .. controls +(-.5,-.5) and +(.5,0) .. (-2,-1.5);
\pgfsetlinewidth{10*\pgflinewidth}
\draw[white]  (-2,-.5) .. controls +(.5,-.5) and +(-.5,0) .. (0,-1.5);
\pgfsetlinewidth{.1*\pgflinewidth}
\draw (-2,-.5) .. controls +(.5,-.5) and +(-.5,0) .. (0,-1.5);

\draw[->] (0,-1.5)   .. controls +(1.5,0) and +(-.5,0) .. (3,-1.5);
\draw[->] (3,-.5)   .. controls +(-1,0) and +(2,-2) .. (0,2);

\pgfsetlinewidth{10*\pgflinewidth}
\draw[white] (0,-.5)  .. controls +(.5,.5) and +(-.5,-1) .. (2,2);
\pgfsetlinewidth{.1*\pgflinewidth}
\draw (0,-.5)  .. controls +(.5,.5) and +(-.5,-1) .. (2,2);

\draw (-2,-.5)   .. controls +(-.5,.5) and +(.5,-1) .. (-4,2);
\pgfsetlinewidth{10*\pgflinewidth}
\draw[white] (-2,-1.5)   .. controls +(-2,0) and +(-2,-2) .. (-2,2);
\pgfsetlinewidth{.1*\pgflinewidth}
\draw (-2,-1.5)   .. controls +(-2,0) and +(-2,-2) .. (-2,2);
\end{scope}
\end{tikzpicture}
\caption{The diagram $K_{h}$.}
\label{fig:homogeneousdiagram}
\end{figure}

\subsubsection*{Almost-positive links and a result of Tagami}

In this subsection, we will prove Corollary \ref{cor:AbeTagami}. This will give us a class of diagrams where \eqref{combinatorialboundons} is sharp while the lower bound in \eqref{KawCavbound} is not. The reader may think that these are the only examples of such diagrams, but this is false. There is a large class of diagrams where \eqref{combinatorialboundons} is sharp while the lower bound in \eqref{KawCavbound} is not. Examples of diagrams belonging to this family, which are not almost-positive, are given in \cite[Chapter 4, Subsection 2.5]{Thesis1} and in the last part of this section.

First we need to recall a result due to Stoimenow \cite[Corollary 5 and the proof of Theorems 5 and 6]{Stoimenov11}, we quote it as in \cite[Theorem 5.1]{AbeTagami17}.

\begin{theorem}[\cite{Stoimenov11}]\label{thm:Stoimenov}
Let $D$ be an almost-positive diagram of a non-split link $L$ with a negative crossing $p$. 
\begin{enumerate}
\item\label{Stoimenov:case1} If there is no (positive) crossing joining the same two Seifert circles of $D$ as the circles which are connected by the negative crossing $p$, we have $g_3(L)=g_3(D)$. 
\item\label{Stoimenov:case2} If there is a (positive) crossing joining the same two Seifert circles of $D$ as the circles which are connected by the negative crossing $p$, we have $g_3(L)=g_3(D)-1$. 
\end{enumerate}
Where $g_3 (D)$ is the genus of the Seifert surface of $D$ (i.e. the surface defined from $D$ using the Seifert algorithm).\qed
\end{theorem}

\begin{proof}[Proof of Corollary \ref{cor:AbeTagami}]
Notice that in Theorem \ref{thm:Stoimenov} we have that: in case \eqref{Stoimenov:case1} $\delta_{-}(D) = 1$, while in case \eqref{Stoimenov:case2} $\delta_{-}(D) = 0$. Thus we can re-write the result of Stoimenov as follows
\[ g_{3}(L) = g_{3}(D)  - (1 - \delta_{-}(D)). \]
In terms of the Euler characteristic we have
\[ -\chi_3(L) = - \chi (D) - 2 (1- \delta_{-} (D)) = n(D) - V(D) + 2\delta_{-}(D) - 2 = \]
\[ = w(D) - V(D) + 2 \delta_{-}(D),\]
where $\chi_{3}(L)$ (resp. $\chi_{4}(L)$) is the minimum Euler characteristic of a surface bounding $L$ in $\mathbb{S}^3$ (resp. in $\mathbb{B}^4$). By property \eqref{case:boundonchi} in Proposition \ref{Prop:sproperties}, and removing a disk from a surface realizing $\chi_{3}(L)$ (resp. $\chi_{4}(L)$), we obtain
\[ s(L,\mathbb{F}) \leq - \chi_{4}(L)  + 1\leq -\chi_{3}(L)  + 1 = w(D) - V(D) + 2 \delta_{-}(D) +1.\]
Recall that $D$ is, by hypothesis, an almost-positive diagram. As a consequence we have $V_{-}(D) = \ell_{-}(D) = 0$. Hence \eqref{combinatorialboundons} can be read as
\[w(D) - V(D) + 2 \delta_{-}(D) +1 \leq s(L,\mathbb{F}).\]
By writing $\chi_{3}(L)$ and $\chi_{4}(L)$ in terms of $g_{3}(L)$ and $g_{4}(L)$, respectively, we get the desired result.
\end{proof}

\begin{proof}[Proof of Theorem \ref{thm:independence}]
It is easy to see that the lower bound in \eqref{KawCavbound} takes the value $w(D) - V(D) + 1$ on each almost-positive diagram. In particular, on each diagram satisfying condition \eqref{Stoimenov:case1} in Theorem \ref{thm:Stoimenov} the bound is not sharp. On the other hand, previously in this section we gave two families of diagrams for which \eqref{KawCavbound} is sharp while \eqref{combinatorialboundons} is not.
\end{proof}

\subsubsection*{Other diagrams}
Finally, we wish to exhibit another family of examples where \eqref{combinatorialboundons} is sharper than \eqref{KawCavbound}. In this case \eqref{combinatorialboundons} cannot be simply recovered from Tagami and Abe-Tagami results summarised in Corollary \ref{cor:AbeTagami}. This family of diagrams is depicted in Figure \ref{fig:lastexample} (consider $r > 0$ and $t, k < 0$) and the corresponding quantities needed to compute the bounds, as well as the values of the bounds, are reported in Table \ref{tab:finalexample}.
\begin{figure}[h]
\centering
\begin{tikzpicture}[scale = .2]

\draw (24,-2.5) rectangle (26,2.5);

\draw (9,-2.5) rectangle (11,2.5);
\draw (14,-1) rectangle (19,1);

\draw[thick, <-] (25.5,2.5) .. controls +(0,5) and +(0,5) .. (9.5,2.5);
\draw[thick,<-] (25.5,-2.5) .. controls +(0,-5) and +(0,-5) .. (9.5,-2.5);

\draw[thick,<-] (10.5,2.5) .. controls +(2,2) and +(-1,0) .. (14,.5);
\draw[thick, <-] (10.5,-2.5) .. controls +(2,-2) and +(-1,0) .. (14,-.5);

\draw[thick, ->] (24.5,2.5) .. controls +(0,2) and +(5,5) .. (19,.5);
\draw[thick,->] (24.5,-2.5) .. controls +(0,-2) and +(5,-5) .. (19,-.5);

\node[rotate = 90]  at (25,0) {$2k $-$ 1$};
\node[rotate = 90]  at (10,0) {$2r +1$};
\node  at (16.5,0) {$2t + 1$};

\draw[fill] (0.5,2.1) .. controls +(.25,0) and +(-.2,0) .. (.75,0) .. controls +(-.25,0) and +(.2,0) .. (0.5,2.1);
\draw[fill] (0.5,-2.1) .. controls +(.25,0) and +(-.2,0) .. (.75,0) .. controls +(-.25,0) and +(.2,0) .. (0.5,-2.1);

\draw[fill] (-10.5,2.1) .. controls +(-.25,0) and +(.2,0) .. (-10.75,0) .. controls +(.25,0) and +(-.2,0) .. (-10.5,2.1);
\draw[fill] (-10.5,-2.1) .. controls +(-.25,0) and +(.2,0) .. (-10.75,0) .. controls +(.25,0) and +(-.2,0) .. (-10.5,-2.1);

\node  at (-12,0) {$2r$};
\node  at (2.5,0) {$-2k$};

\draw[thick, blue] (-5,2) -- (-5,-2);
\node[left] at (-5,0) {$-$};

\draw[thick, blue] (0,-2) -- (0,-1);
\draw[thick, blue] (0,2) -- (0,1);

\draw[thick, red] (-10,-2) -- (-10,-1);
\draw[thick, red] (-10,2) -- (-10,1);

\draw[thick, blue] (-5,2) -- (0,2);
\node[above] at (-2.5,2) {$-$};

\draw[thick, blue] (-5,-2) -- (0,-2);
\node[above] at (-2.5,-2) {$-$};

\draw[thick, red] (-5,2) -- (-10,2);
\node[above] at (-7.5,2) {+};

\draw[thick, red] (-5,-2) -- (-10,-2);
\node[below] at (-7.5,-2) {+};

\draw[fill,green] (-5,2) circle (.2);
\draw (-5,2) circle (.2);

\draw[fill,green] (-5,-2) circle (.2);
\draw (-5,-2) circle (.2);

\draw[fill,blue] (0,2) circle (.2);
\draw (0,2) circle (.2);

\draw[fill,blue] (0,-2) circle (.2);
\draw (0,-2) circle (.2);

\draw[fill,red] (-10,2) circle (.2);
\draw (-10,2) circle (.2);

\draw[fill,red] (-10,-2) circle (.2);
\draw (-10,-2) circle (.2);

\end{tikzpicture}
\caption{The diagram $D(r,k,t)$, and its simplified Seifert graph in the case $r>0$ and $k,t <0$. The absolute value of the numbers on the boxes indicate the number of crossings (of the same sign as the number) to be placed in the box.}\label{fig:lastexample}
\end{figure}
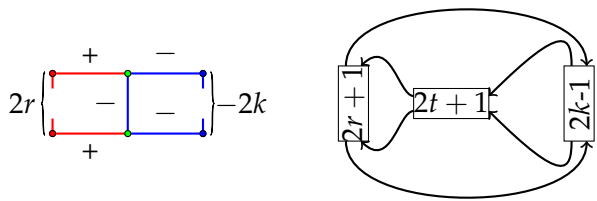

\begin{table}[h]
\centering
\begin{tabular}{ccc}
Quantities & Value\\
\hline\hline
$w_{\ }$  & $2(r+k+t) + 1$   \\
$V_{\ }$  & $2r - 2k + 2$ \\
$V_-$     & $-2k$ \\
$s_+$     & $-2k + 1$ \\
$s_-$     & $2r + 1$ \\
$\ell_s$  & $1$ \\
$\ell_{\ }$  & $1$ \\
$\ell_-$  & $0$ \\
$\delta_-$& $1$ \\\hline
\eqref{Lobbbound} upper & $ 2t + 2$\\
\eqref{combinatorialboundons}& $ 2t + 2$\\
\eqref{KawCavbound} & $2t$\\
\end{tabular}
\caption{Comparing the bounds on $s(D(r,t,k),\mathbb{F})$, for $r>0$ and $t,k < 0$.}\label{tab:finalexample}
\end{table}

In particular, we have proved the following.

\begin{proposition}
The $s$-invariant of the knot represented by the diagram $D(r,t,k)$, for $r > 0$ and $t, k < 0$, is equal to $2t + 2$.\qed
\end{proposition}

\begin{rem*}
The knot represented by the diagram in Figure \ref{fig:lastexample} is a 2-bridge knot. In particular, in this case the $s$-invariant coincide with the signature. However, it is easy to produce non-2-bridge (and in general non-thin) examples where \eqref{combinatorialboundons} is sharp (e.g. the diagram of the $9_{42}$ in Figure \ref{fig:942}).
\end{rem*}

\section{Final remarks and open questions}

In this last section we wish to comment the examples above, and to discuss some open questions. The first, and most natural, question one could ask is the following: for each link type $\lambda$ there always exists a diagram $D$ such that at least one among the lower bounds given in \eqref{KawCavbound}, \eqref{Lobbbound}, and \eqref{combinatorialboundons}, is sharp?

Given that all these combinatorial bounds strongly depend on the diagram, one might think that this question may have a positive answer. However, surprisingly enough, it is simple to show that the answer is negative. The bounds  given in \eqref{KawCavbound}, \eqref{Lobbbound}, and \eqref{combinatorialboundons} do not depend on the characteristic of the field $\mathbb{F}$. On the other hand, the $s$-invariant depends on $char(\mathbb{F})$; for example, it is known that $s(14_{n19265};\mathbb{F}_{2}) < s(14_{n19265};\mathbb{Q})$ (see \cite[Remark 6.1]{LipshitzSarkar12}). Thus, for each diagram $L$ of the knot $14_{n19265}$ the lower bounds in \eqref{KawCavbound}, \eqref{Lobbbound}, and \eqref{combinatorialboundons} cannot be sharp over $\mathbb{Q}$.

\begin{rem*}
If we consider the knot $m(14_{n19265})$, then 
\[s\left(m(14_{n19265});\mathbb{F}_{2}\right)  = -s(14_{n19265};\mathbb{F}_{2}) > - s(14_{n19265};\mathbb{Q}) = s\left(m(14_{n19265});\mathbb{Q}\right). \]
Thus the lower bounds in \eqref{KawCavbound}, \eqref{Lobbbound}, and \eqref{combinatorialboundons} cannot be sharp over $\mathbb{F}_2$, for each diagram of $m(14_{n19265})$. This reasoning shows that, if $s(\kappa;\mathbb{F}) \ne s(\kappa;\mathbb{K})$, for some fields $\mathbb{F}$ and $\mathbb{K}$ and some knot $\kappa$, then the above mentioned combinatorial bounds cannot be sharp neither over $\mathbb{F}$ nor over $\mathbb{K}$.
\end{rem*}

Given a field $\mathbb{F}$, either $s(14_{n19265};\mathbb{F}) \ne s_{}(14_{n19265};\mathbb{F}_2)$ or $s(14_{n19265};\mathbb{F}) = s(14_{n19265};\mathbb{F}_2)$, which is different from $ s(14_{n19265};\mathbb{Q})$. Thus we may conclude that there is no field $\mathbb{F}$ such that the combinatorial bounds given in either \eqref{KawCavbound}, \eqref{Lobbbound}, or \eqref{combinatorialboundons} are sharp for all link types\footnote{Here we mean that for each link type exists a diagram such that the given bound is sharp.}.

\begin{rem*}
It is still unknown (at least to the best of the author's knowledge) if for each pair of fields, say $\mathbb{F}$ and $\mathbb{K}$, there exists a link-type $\lambda$ such that $s(\lambda;\mathbb{F}) \ne s(\lambda; \mathbb{K})$ (cf. \cite[Question 6.1]{LipshitzSarkar12}).
\end{rem*}

So, we are naturally led to the following question, which remains still open.

\begin{question}
Is there a combinatorial (either upper or lower) bound on the value of $s(\lambda;\mathbb{F})$, which depends on the characteristic of the field $\mathbb{F}$, and is sharp for all link types?
\end{question}

Since the only obstruction we provided to the sharpness of the combinatorial bounds in \eqref{KawCavbound}, \eqref{Lobbbound}, and \eqref{combinatorialboundons} is the difference between the value of the $s$-invariant for different fields, it makes sense to ask the following question.

\begin{question}\label{question:sharpnessforsindeponF}
Given an oriented link type $\lambda$ such that $s(\lambda; \mathbb{F})$ does not depend on $\mathbb{F}$, is there a diagram $D$ representing $\lambda$ such that one among the lower bounds in \eqref{KawCavbound}, \eqref{Lobbbound}, and \eqref{combinatorialboundons} is sharp?
\end{question}

First, let us notice that the family of links such that $s(\lambda; \mathbb{F})$ does not depend on $\mathbb{F}$ is non empty. For example, it contains all $Kh$-pseudo-thin links. To be precise is strictly larger than the family of $Kh$-pseudo-thin links. An example of non-pseudo-thin knot whose $s$-invariant does not depend on the characteristic of the field is the $9_{42}$. In this case, the integral Khovanov homology contains only $2$-torsion (see \cite{KnotAtlas}), thus is equal in every field of characteristic different than $2$. The $s$-invariant over $\mathbb{F}_2$ can be computed using \texttt{knotkit} (\cite{Knotkit}), and is $0$. Since the $9_{42}$ is a knot, the $s$-invariant can be read directly from Bar-Natan homology (\cite[Chapter 2]{Thesis1}). Moreover, in this case Bar-Natan homology can be directly computed from Khovanov homology (\cite[Appendix B]{Thesis1}), leading us to $s(9_{42};\mathbb{F}) = 0$, for $char(\mathbb{F}) \ne 2$.

\begin{figure}
\centering
\includegraphics[scale=.2]{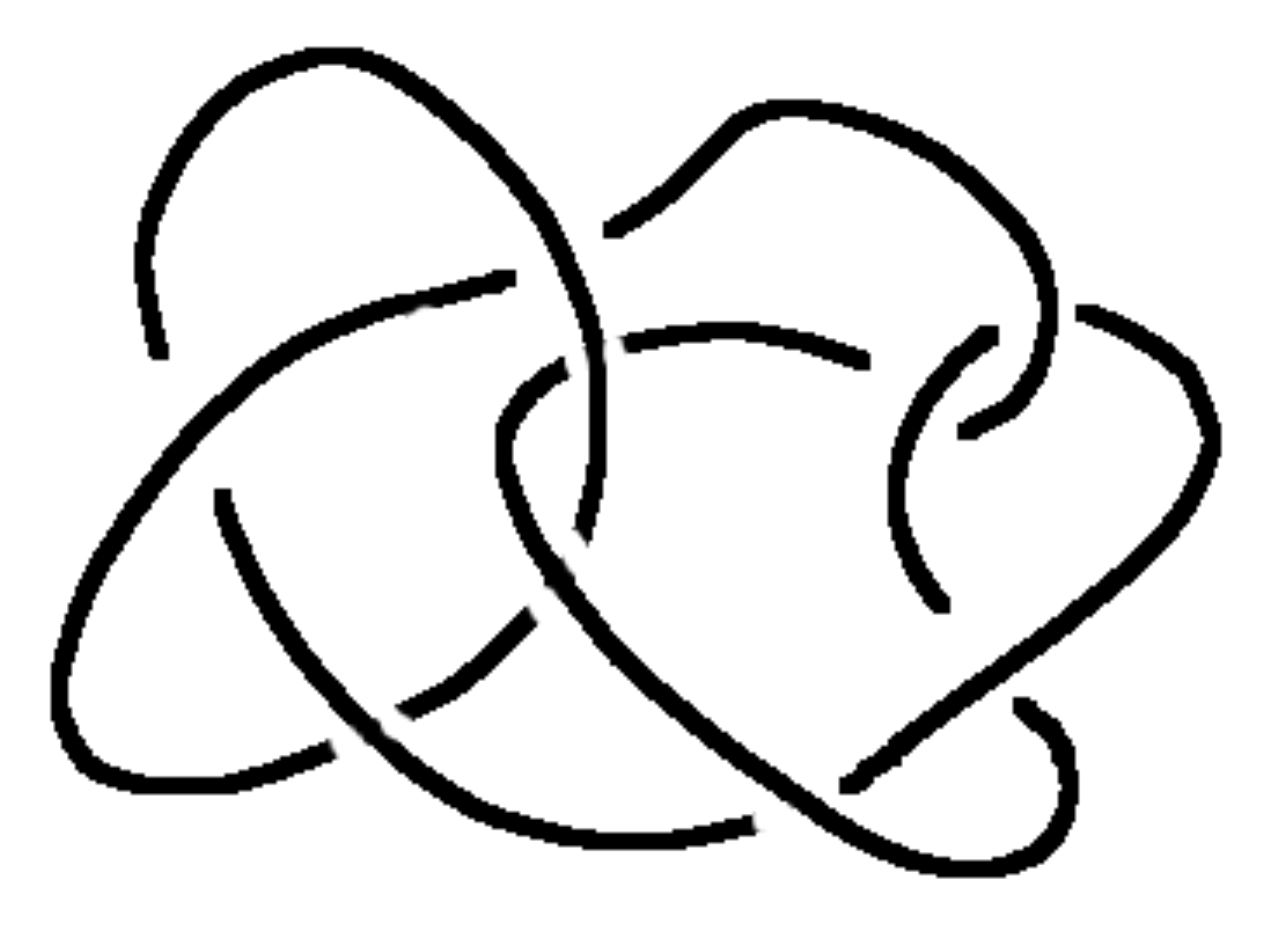} 
\caption{A diagram of $9_{42}$.}
\label{fig:942}
\end{figure}

\begin{rem*}
The diagram in Figure \ref{fig:942} is such that the bound in Equation \eqref{combinatorialboundons} is sharp. In particular, this gives an example of diagram of a non-pseudo-thin knot for  which the bound in Equation \eqref{combinatorialboundons} is sharp.
\end{rem*}

Question \ref{question:sharpnessforsindeponF} remains still open. However, something can be said in the case where the lower bound is the one provided by \eqref{combinatorialboundons}. This bound comes from an Bennequin $s$-inequality plus an estimate on the value of the $c$-invariants. Thus, an obstruction for the bound in \eqref{combinatorialboundons} to be sharp on a certain diagram $L$ is the sharpness of the Bennequin $s$-inequality on $D$. It follows that an affirmative answer for Question \ref{question:sharpnessforsindeponF}, restricted to the bound coming from \eqref{combinatorialboundons}, is subordinate an affirmative answer to the following question.

\begin{question}\label{question:sharpnessofBennequinforsindeponF}
Given an oriented link type $\lambda$ such that $s(\lambda; \mathbb{F})$ does not depend on $\mathbb{F}$, is there a diagram $D$ representing $\lambda$ such that the Bennequin $s$-inequality in Proposition \ref{Prop:Bennequins} is sharp?
\end{question}

We notice that, even if in Question \ref{question:sharpnessforsindeponF} it was necessary to ask the independence of the $s$-invariant from the field, there is no need to restrict to this setting in the case of the Bennequin $s$-inequality. In all examples known to the author the inequality is sharp.

\begin{question}\label{question:sharpnessofBennequin}
Is the Bennequin $s$-inequality in Theorem \ref{thm:main2} sharp for all transverse links?
\end{question}

The author, in \cite{TransFromKhType17}, proved some sufficient conditions on a link type $\lambda$ to be such that the Bennequin $s$-inequality is sharp for each transverse representative of $\lambda$. However, we are still far from knowing the complete answer to Question \ref{question:sharpnessofBennequin}. 

We remark that if Question \ref{question:sharpnessofBennequin} admits a positive answer, then the $c$-invariants are ineffective transverse invariants. More precisely, if the Bennequin $s$-in\-e\-qua\-li\-ty is sharp then the$ c$-invariants can be written as an affine combination of a classical invariant ($sl(B)$) and a topological invariant ($s(B;\mathbb{F})$). Thus, the $c$-invariants cannot be effective. In particular, the (vanishing of the homology class of the) Plamenevskaya invariant will be a non-effective invariant.

The $s$-invariant is part of a more general family of invariants, called slice-torus invariants, which can be defined as follows. We remark that our definition is normalized differently from the original definition due to L. Lewark.

\begin{definition}
A \emph{slice-torus invariant} is an integer link invariant $\nu$ such that:
\begin{itemize}
\item[1.] $\nu$ defines an homomorphism from the concordance group of knots to $\mathbb{Q}$;
\item[2.] $2g_{4}(L) \geq \vert \nu(L) \vert$, where $g_{4}$ denotes the slice genus, for each knot $L$;
\item[3.] For each torus knot $T(p,q)$ we have that
\[ \nu(T(p,q)) =  (p-1)(q-1)\]
\end{itemize}
\end{definition}

The family of the slice torus invariants, as defined above, include (a suitable normalization of) the $\tau$ invariant from knot Floer homology \cite{OzsvathSzabo03a}, as well as $\mathfrak{sl}_n$ generalizations of the $s$-invariants \cite{LobbLewark15}. In \cite{Kawamura15} T. Kawamura proved that in \eqref{KawCavbound} the $s$-invariant can be replaced with any slice-torus invariant.

\begin{question}\label{question:slice-torus}
Let $D$ be an oriented link diagram representing the link $L$. Does the bound
\[\nu(L) \geq w(D) - V(D) + 2 V_-(D) - 2 \ell_{-}(D) + 2\delta_{-}(D) + 1\]
hold for any slice-torus invariant $\nu$?
\end{question}

The point of Question \ref{question:slice-torus} is to understand how much \eqref{combinatorialboundons} is a ``topological bound''. Personally, the author believes that the answer is negative, as he expects the bound in \eqref{combinatorialboundons} to be ``intrinsically Khovanov-theoretic''. However, there is the possibility that adapting the techniques from Stoimenow \cite{Stoimenov11} and Kawamura \cite{Kawamura15} a positive answer to Question \ref{question:slice-torus} can be found.

\bibliography{Bibliography.bib}
\bibliographystyle{plain}

\end{document}